\documentclass{article}

\usepackage[utf8]{inputenc} 
\usepackage[T1]{fontenc}    
\usepackage{hyperref}       
\usepackage{url}            
\usepackage{booktabs}       
\usepackage{amsfonts}       
\usepackage{nicefrac}       
\usepackage{microtype}      
\usepackage{lipsum}
\usepackage{graphicx}
\graphicspath{ {./images/} }

\usepackage{abstract}

\usepackage{amsmath}
\usepackage{amsthm}

\newtheorem{lemma}{Lemma}
\newtheorem{theorem}{Theorem}
\newtheorem{proposition}{Proposition}
\newtheorem{definition}{Definiton}
\newtheorem{corollary}{Corollary}
\newtheorem{remark}{Remark}

\newtheorem*{theorem'}{Theorem A}
\newtheorem*{corollary'}{Corollary A}
\newtheorem*{theorem''}{Theorem B}
\newtheorem*{corollary''}{Corollary B}

\title{The Lagrange Problem from the Viewpoint of Toric Geometry}

\author{
 Xiuting Tang \\
  School of Mathematics, Shandong University\\
  Jinan, Shandong, 250100, China\\
  \texttt{tangxiuting@mail.sdu.edu.cn}
   }

\begin{document}
\maketitle
\begin{abstract}
In this paper, I mainly prove the following results.  For every energy value below the minimum of the first, second and third critical value, each bounded component of the regularized energy hypersurface of the Lagrange problem with $0<m_2\leq m_1\leq 9m_2$, $m_1\geq\frac{\epsilon}{2}>0$ and $m_2\geq\frac{3\epsilon}{8}$ arises as the boundary of a strictly monotone toric domain, which is dynamically convex as a corollary. 

For the Euler problem as a special case of the Lagrange problem, when the energy $c<-m_1-m_2$, the bounded component around the fixed center $e$ of the regularized energy hypersurface of the Euler problem with two fixed points $e$ and $m$ of masses $m_1$ and $m_2$ respectively satisfying $m_1>0, m_2\leq 0$ and $m_1\geq|m_2|$ arises as the boundary of a convex toric domain. Together with Gabriella Pinzari's result, when the energy is less than the critical value, the toric domain $X_{\Omega_{m_2}}$ defined above is concave for $m_2\geq 0$, convex for $m_2\leq 0$.
\end{abstract}

\section{Introduction}

The Lagrange problem is the problem of two fixed centers  adding a centrifugal force from the middle of the two fixed centers. Setting the two fixed centers as $e=(-\frac{1}{2},0)$ and $m=(\frac{1}{2},0)$, its Hamiltonian function is

\begin{equation}H_{\epsilon}(q,p)=T(p)-U_{\epsilon}(q).\end{equation}\label{H of Lagrange}
where
$$T(p)=\frac{1}{2}|p|^2.$$
$$U_{\epsilon}:\mathbb{R}^2\backslash\{(\pm\frac{1}{2},0)\}\rightarrow \mathbb{R},
q\rightarrow \frac{m_1}{\sqrt{(q_1+\frac{1}{2})^2+q_2^2}}+\frac{m_2}{\sqrt{(q_1-\frac{1}{2})^2+q_2^2}}+\frac{\epsilon}{2}|q|^2,$$
$p=(p_1,p_2)^T, q=(q_1,q_2)^T$
and $m_1,m_2,\epsilon\in \mathbb{R}$.

If we cancel the centrifugal force and let one mass to be zero, i.e. $\epsilon=0$ and $m_2=0$, (\ref{H of Lagrange}) becomes the Hamiltonian of Kepler problem. If we just cancel the centrifugal force, i.e. $\epsilon=0$, (\ref{H of Lagrange}) is the Hamiltonian of Euler problem. If we set $\epsilon=1$ and $m_1=m_2=\frac{1}{2}$, it has the same potential energy as the restricted three body problem, and the whole system is the restricted three body problem subtracting a Coriolis force.
So we can consider the Lagrange problem as a perturbation of the Kepler problem and Euler problem and it also reflect some information about the restricted three body problem near the boundary of the Hill's region.

It was first observed by Lagrange \cite{Lagrange} that the problem of two fixed centers remain integrable if one adds an elastic force acting from the midpoint of the two masses. In case the two masses are equal the elastic force can be interpreted as the centrifugal force. We refer to the paper by \cite{Hiltebeitel} for a comprehensive treatment which forces one can add to the problem of two fixed centers while still keeping the problem completely integrable. As a special case, the Lagrange problem is an integrable system. The technique to show that the Lagrange problem is integral in \cite{Hiltebeitel} is the elliptic coordinate. Using the elliptic coordinate, the Lagrange problem can be regularized and separated to two Hamiltonian systems. 

Since the Lagrange problem is integrable and can be separated, we can define the momentum map and study its toric domain. Toric domain is an important concept in symplectic geometry, especially in symplectic embedding theory \cite{Choi, Cristofaro-Gardiner}, which is a very active research area with a lot of striking new results inspired by the landmark paper \cite{McDuff}. The first natural example of a concave toric domain discovered by \cite{Ramos} is the Lagrangian bidisk, and then \cite{Ferreira}, \cite{Frauenfelder}. Our paper is mainly inspired by \cite{Frauenfelder}, \cite{Hiltebeitel} and \cite{Pinzari}. 

We will prove the following theorem A in section 4. Without loss of generality, we assume $m_1\geq m_2$ in theorem A, the case when $m_1<m_2$ can be easily obtained by change the positions of the two fixed centers.
\begin{theorem'}
Assume $0<m_2\leq m_1\leq 9m_2$, $m_1\geq\frac{\epsilon}{2}>0$ and $m_2\geq\frac{3\epsilon}{8}$, for every energy value below the minimum of the first, second and third critical value, i.e. $c<c_0$, each bounded component of the regularized energy hypersurface of the Lagrange problem arises as the boundary of a strictly monotone toric domain.
\end{theorem'}

\begin{corollary'}
Assume $0<m_2\leq m_1\leq 9m_2$, $m_1\geq\frac{\epsilon}{2}>0$ and $m_2\geq\frac{3\epsilon}{8}$, for every energy value below the minimum of the first, second and third critical value, i.e. $c<c_0$, each bounded component of the regularized energy hypersurface of the Lagrange problem is dynamically convex.
\end{corollary'}

The Euler problem can be treated as a special case of the Lagrange problem when $\epsilon=0$ in (\ref{H of Lagrange}). About the Euler problem, according to the result of Gabriella Pinzari in \cite{Pinzari} and my estimate (\ref{estimate 1}) in next section, it is a concave toric domain below the critical value for positive masses $m_1>0$ and $m_2>0$. For negative masses, using Gabriella Pinzari's method, we can prove the following theorem B in section 4 that only when the energy $c<0$, the orbit of the Euler problem is in the bounded Hill's region. The toric domain of the bounded Hill's region near the big mass is convex under the conditions $m_1>0, m_2\leq 0$ and $m_1>|m_2|$.

\begin{theorem''}
When the energy $c<-m_1-m_2$, the bounded component around the fixed center $e$ of the regularized energy hypersurface of the Euler problem with two fixed points $e$ and $m$ of masses $m_1$ and $m_2$ respectively satisfying $m_1>0, m_2\leq 0$ and $m_1\geq|m_2|$ arises as the boundary of a convex toric domain.
\end{theorem''}

Combining Gabriella Pinzari's result in \cite{Pinzari} and theorem B, we have the following corollary for Euler problem.
\begin{corollary''}
When the energy is less than the critical value, i.e. $c<c_o$, the toric domain $X_{\Omega_{m_2}}$ defined for the bounded component around the fixed center $e$ of the regularized energy hypersurface of the Euler problem  with fixed centers $e$ and $m$ of masses $m_1$ and $m_2$ respectively satisfying $m_1\geq|m_2|$ is concave for $m_2\geq 0$, convex for $m_2\leq 0$.
\end{corollary''}

My paper is organized as follows. In section 2, we discuss the critical points of the Lagrange problem and the Euler problem under some conditions of $m_1$, $m_2$ and $\epsilon$. Since the Hamiltonian of the Lagrange problem has singularities at the two big bodies, in section 3, we give the regularization of the Lagrange problem. In section 4, we define the moment map and the toric domain of the Lagrange problem and prove theorem A of this paper. As a supplement of section 4, in section 5, we give a simple definition of toric domain and list some properties already known before our paper. In section 5, using Gabriella Pinzari's method, we prove theorem B of this paper.

\section{Critical points of the Lagrange problem and the Euler problem}

In this section, we discuss the critical points of the Hamiltonian $H_{\epsilon}$ given by (\ref{H of Lagrange}) under some conditions of $m_1$, $m_2$ and $\epsilon$.  

We can immediately observe from Hamiltonian (\ref{H of Lagrange}) that the projection map $\pi:\mathbb{R}^4=\mathbb{R}^2\times\mathbb{R}^2\rightarrow\mathbb{R}^2$ given by $(p, q)\mapsto q$ induces a bijection between the critical points of $H_{\epsilon}$ and that of $U_{\epsilon}$.
$$\pi\big|_{crit(H_{\epsilon})}: crit(H_{\epsilon})\rightarrow crit(U_\epsilon)$$
By a direct computation, we know that the inverse map for a critical point $(q_1, q_2)\in crit(U_{\epsilon})$ is given by
$$\big(\pi\big|_{crit(H_{\epsilon})}\big)^{-1}(q_1,q_2)=(0,0,q_1,q_2).$$
At each fixed critical point $l\in crit(U_{\epsilon})$, note $$L=\pi\big|^{-1}_{crit(H_{\epsilon})}(l)\in crit(H_{\epsilon}),$$ we have \begin{equation}\label{H=-U}H_{\epsilon}(L)=-U_{\epsilon}(l).\end{equation}

For Lagrange problem with $m_1>0, m_2>0,\epsilon>0$, we get the following two lemmas.

\begin{lemma}
The Lagrange problem with $m_1>0,m_2>0,\epsilon>0$ has five critical points. There are three critical points $l_1, l_2, l_3$ in the x-axis. If $l_1, l_2, l_3$ are non-degenerate, they are saddle points.  There are two maxima $l_4, l_5$ symmetric with respect to x-axis.  
\end{lemma}
\begin{proof}
The partial derivative and second partial derivative of $V(q)=-U_{\epsilon}(q)$ with respect to $q_1$ and $q_2$ are
\begin{equation}\label{D1U}
\frac{\partial V}{\partial q_1}=\frac{m_1(q_1+\frac{1}{2})}{\big((q_1+\frac{1}{2})^2+q_2^2\big)^{\frac{3}{2}}}+\frac{m_2(q_1-\frac{1}{2})}{\big((q_1-\frac{1}{2})^2+q_2^2\big)^{\frac{3}{2}}}-\epsilon q_1,
\end{equation}
\begin{equation}\label{D2U}
\frac{\partial V}{\partial q_2}=\frac{m_1q_2}{\big((q_1+\frac{1}{2})^2+q_2^2\big)^{\frac{3}{2}}}+\frac{m_2 q_2}{\big((q_1-\frac{1}{2})^2+q_2^2\big)^{\frac{3}{2}}}-\epsilon q_2.
\end{equation}
\begin{equation}\label{DD1U}
\begin{aligned}
\frac{\partial^2V}{\partial q_1^2}=&\frac{m_1}{\big((q_1+\frac{1}{2})^2+q_2^2\big)^{\frac{3}{2}}}-\frac{3m_1(q_1+\frac{1}{2})^2}{\big((q_1+\frac{1}{2})^2+q_2^2\big)^{\frac{5}{2}}}\\&+\frac{m_2}{\big((q_1-\frac{1}{2})^2+q_2^2\big)^{\frac{3}{2}}}-\frac{3m_2(q_1-\frac{1}{2})^2}{\big((q_1-\frac{1}{2})^2+q_2^2\big)^{\frac{5}{2}}}-\epsilon,
\end{aligned}
\end{equation}
\begin{equation}\label{DD2U}
\begin{aligned}
\frac{\partial^2V}{\partial q_2^2}=&\frac{m_1}{\big((q_1+\frac{1}{2})^2+q_2^2\big)^{\frac{3}{2}}}-\frac{3m_1q_2^2}{\big((q_1+\frac{1}{2})^2+q_2^2\big)^{\frac{5}{2}}}\\&+\frac{m_2}{\big((q_1-\frac{1}{2})^2+q_2^2\big)^{\frac{3}{2}}}-\frac{3m_2q_2^2}{\big((q_1-\frac{1}{2})^2+q_2^2\big)^{\frac{5}{2}}}-\epsilon,
\end{aligned}
\end{equation}

Firstly, we consider the critical points with $q_2\neq 0$ and note such critical point by $l_j, j=4,5$. By (\ref{D1U}) and (\ref{D2U}), $L_j$ satisfies
$$\frac{m_1(q_1+\frac{1}{2})}{\big((q_1+\frac{1}{2})^2+q_2^2\big)^{\frac{3}{2}}}+\frac{m_2(q_1-\frac{1}{2})}{\big((q_1-\frac{1}{2})^2+q_2^2\big)^{\frac{3}{2}}}-\epsilon q_1=0,$$
$$\frac{m_1}{\big((q_1+\frac{1}{2})^2+q_2^2\big)^{\frac{3}{2}}}+\frac{m_2}{\big((q_1-\frac{1}{2})^2+q_2^2\big)^{\frac{3}{2}}}-\epsilon=0.$$
They are equivalent to 
\begin{equation}\label{equ of Lj}
\frac{m_2}{\big((q_1-\frac{1}{2})^2+q_2^2\big)^{\frac{3}{2}}}=\frac{m_1}{\big((q_1+\frac{1}{2})^2+q_2^2\big)^{\frac{3}{2}}}=\frac{\epsilon}{2}.
\end{equation}
Since $\epsilon>0$, they have two solutions
$$\tilde{q}_1=\frac{m_1^{\frac{2}{3}}-m_2^{\frac{2}{3}}}{2^{\frac{1}{2}}\epsilon^{\frac{2}{3}}}.$$
$$\tilde{q}_2=\pm\sqrt{\bigg(\frac{2m_1}{\epsilon}\bigg)^{\frac{2}{3}}-\bigg(q_1+\frac{1}{2}\bigg)^2}$$
By (\ref{DD1U}),(\ref{DD2U}) and (\ref{equ of Lj}), we get
\begin{equation}\label{DD1V}\frac{\partial^2 V}{\partial q_1^2}(l_j)=-\frac{3m_1(q_1+\frac{1}{2})^2}{\big((q_1+\frac{1}{2})^2+q_2^2\big)^{\frac{5}{2}}}-\frac{3m_2(q_1-\frac{1}{2})^2}{\big((q_1-\frac{1}{2})^2+q_2^2\big)^{\frac{5}{2}}}<0,\end{equation}
\begin{equation}\label{DD2V}\frac{\partial^2 V}{\partial q_2^2}(l_j)=-\frac{3m_1q_2^2}{\big((q_1+\frac{1}{2})^2+q_2^2\big)^{\frac{5}{2}}}-\frac{3m_2q_2^2}{\big((q_1-\frac{1}{2})^2+q_2^2\big)^{\frac{5}{2}}}<0.\end{equation}
From (\ref{D1U}), we get
$$\frac{\partial^2V}{\partial q_2\partial q_1}=-\frac{3m_1(q_1+\frac{1}{2})q_2}{\big((q_1+\frac{1}{2})^2+q_2^2\big)^{\frac{5}{2}}}-\frac{3m_2(q_1-\frac{1}{2})q_2}{\big((q_1-\frac{1}{2})^2+q_2^2\big)^{\frac{5}{2}}}$$
By a simple direct computation, we get 
$$\det\left ( \begin{matrix}
\frac{\partial^2V}{\partial q_1^2} & \frac{\partial^2V}{\partial q_1q_2} \\
\frac{\partial^2V}{\partial q_1q_2} & \frac{\partial^2V}{\partial q_2^2} \\
\end{matrix} \right )=\frac{9m_1m_2q_2^2}{{\big((q_1+\frac{1}{2})^2+q_2^2\big)^{\frac{5}{2}}}{\big((q_1-\frac{1}{2})^2+q_2^2\big)^{\frac{5}{2}}}}>0. $$
Together with (\ref{DD1V}) and (\ref{DD2V}), we know that the Hessian of $V(q)$ is negative definite at $q=l_4$ and $q=l_5$. So $l_4=(\tilde{q}_1,\tilde{q}_2)$ and $l_5=(\tilde{q}_1,-\tilde{q}_2)$ are maxima of the potential energy.

Secondly, we consider the critical points in the x-axis, i.e. $q_2=0$.
If $q_1\rightarrow\pm\frac{1}{2}$ or $q_1\rightarrow\pm\infty$, then $V(q)=-U(q)$ all go to $+\infty$. As a result, there are at least three maxima of $H$ restricted to the x-axis $l_1=(\iota_1,0)$, $l_2=(\iota_2,0)$ and $l_3=(\iota_3,0)$ in the x-axis with $-\frac{1}{2}<\iota_1<\frac{1}{2}$, $\iota_2>\frac{1}{2}$ and $\iota_3<-\frac{1}{2}$. 
Note such critical points by $l_i, i=1,2,3$. By (\ref{DD1U}),
\begin{equation}\label{partial U q1q1 on x-axis}
\frac{\partial^2 V}{\partial q_1^2}\bigg|_{q_2=0}=-\frac{2m_1}{|q_1+\frac{1}{2}|^3}-\frac{2m_2}{|q_1-\frac{1}{2}|^3}-\epsilon<0.
\end{equation}
i.e. $V(q_1,0)$ is convex on $(-\infty, -\frac{1}{2})$, $(-\frac{1}{2},\frac{1}{2})$ and $(\frac{1}{2},\infty)$ separately. As a result, $V$ just has three critical points $l_1$, $l_2$ and $l_3$.

By (\ref{partial U q1q1 on x-axis}), $\frac{\partial^2 V}{\partial q_1^2}(l_i)<0, i=1,2,3$. To prove that the collinear critical points $l_1, l_2, l_3$ are saddle points one need to show that
$$\det\left ( \begin{matrix}
\frac{\partial^2V}{\partial q_1^2}(l_i) & \frac{\partial^2V}{\partial q_1q_2}(l_i) \\
\frac{\partial^2V}{\partial q_1q_2}(l_i) & \frac{\partial^2V}{\partial q_2^2}(l_i) \\
\end{matrix} \right )<0, i=1,2,3. $$
Because $U$ is invariant under reflection at the $q_1$-axis and the three collinear critical points are fixed points of this flection, we conclude that
$$\frac{\partial^2V}{\partial q_1q_2}(l_i)=0, i=1,2,3.$$
Since we already have (\ref{DD1U}), it suffices to check that 
$$\frac{\partial^2V}{\partial q_2^2}(l_i) >0, i=1,2,3.$$
Now assume that the collinear Lagrange points are non-degenerate in the sense that the kernel of the Hessian at them is trivial. By the discussion above this is equivalent to the assumption that
$$\frac{\partial^2V}{\partial q_1^2}(l_i)\neq 0, i=1,2,3.$$
Note that the Euler characteristic of the two fold punctured plane satisfies
$$\chi(\mathbb{R}\backslash\{e,m\})=-1,$$ 
where $e=(-\frac{1}{2},0)$, $m=(\frac{1}{2},0)$. Denote by $\nu_2$ the number of maxima of V, by $\nu_1$ the number of saddle points of $V$, and by $\nu_0$ the number of minima of $U$. Because $V=-U$ goes to $-\infty$ at infinity as well as at the singularities $e$ and $m$, it follows from the Poincar\'{e}-Hopf index theorem that 
\begin{equation}\label{sum of nu}\nu_2-\nu_1+\nu_0=\chi(\mathbb{R}\backslash\{e,m\})=-1.\end{equation}
By the first step, we know that $L_4$, $L_5$ are maxima, so that
\begin{equation}\label{nu 2}\nu_2\geq 2.\end{equation}
Since $l_1, l_2, l_3$ are maxima of the restriction of $U$ to the x-axis, it follows that they are either saddle points or maxima of $U$. As a result, 
\begin{equation}\label{nu 0}\nu_0=0.\end{equation}
Combining(\ref{sum of nu}),(\ref{nu 2}), (\ref{nu 0}) 
and the number of non-degenerate critical points
$$\nu_2+\nu_1+\nu_0=5,$$
we conclude that
$$\nu_2=2, \nu_1=3.$$
As a result, $l_1, l_2, l_3$ are saddle points of the potential $V$.
This finishes the proof of the lemma in the non-degenerate case.

\end{proof}

The next lemma tells us that the minimum of the first, second and third critical value depends closely on $m_1, m_2, \epsilon$.
\begin{lemma}\label{lem2}
For Lagrange problem with $m_1>0, m_2>0, \epsilon>0$. Assume $m_1\geq m_2$. When $m_1\geq \frac{\epsilon}{2}, m_2\geq \frac{3\epsilon}{8}$, we have $H_{\epsilon}(L_1)<H_{\epsilon}(L_3)<H_{\epsilon}(L_2)$. When $m_1< \frac{3\epsilon}{8}, m_2\leq \frac{5\epsilon}{24}$, we have
$H_{\epsilon}(L_3)<H_{\epsilon}(L_2)<H_{\epsilon}(L_1).$
\end{lemma}
\begin{proof}
We claim that $0\leq \iota_1<\frac{1}{2}$, $l_1=(\iota_1, 0)$. In fact, for $-\frac{1}{2}<q_1<\frac{1}{2}$, by ($\ref{D1U}$),
$$\frac{\partial V}{\partial q_1}\bigg|_{l_1}=\frac{m_1}{(\iota_1+\frac{1}{2})^2}-\frac{m_2}{(\iota_1-\frac{1}{2})^2}-\epsilon \iota_1=0.$$
If $\iota_1<0$, we have 
$$\frac{m_1}{(\iota_1+\frac{1}{2})^2}<\frac{m_2}{(\iota_1-\frac{1}{2})^2}.$$
Since $m_1\geq m_2$, it requires $(\iota_1+\frac{1}{2})^2>(\iota_1-\frac{1}{2})^2$, that conflicts with $\iota_1<0$.
So it must be $\iota_1\geq 0$.

Firstly, we compare $V(l_2)$ and $V(l_3)$.
Let $\iota>\frac{1}{2}$, $l=(\iota,0)$, then 
$$V(l)=-\frac{m_1}{|l+\frac{1}{2}|}-\frac{m_2}{|l-\frac{1}{2}|}-\frac{\epsilon l^2}{2},$$
$$V(-l)=-\frac{m_1}{|l-\frac{1}{2}|}-\frac{m_2}{|l+\frac{1}{2}|}-\frac{\epsilon l^2}{2},$$
$$V(l)-V(-l)=\frac{m_1-m_2}{|l-\frac{1}{2}|}-\frac{m_1-m_2}{|l+\frac{1}{2}|}.$$
Since $l>\frac{1}{2}$, we have $V(l)<V(-l)$ for all $l>\frac{1}{2}$, and 
\begin{equation}\label{V2<V3}V(l_2)>V(l_3).\end{equation}

Secondly, we compare $V(l_1)$ with $V(l_2)$ and $V(l_3)$.
Let $(t,0), 0\leq t<\frac{1}{2}$ and $(s,0)$ be symmetric point of $(t,0)$ with respect to the point $(\frac{1}{2}, 0)$. Let $(r, 0)$ be the symmetric point of $(t,0)$ with respect to the point $(-\frac{1}{2},0)$. Note $\rho=\frac{1}{2}-t=s-\frac{1}{2}$, $1-\rho=t+\frac{1}{2}=-\frac{1}{2}-r$ ,then
$$V(t)=-\frac{m_1}{1-\rho}-\frac{m_2}{\rho}-\frac{\epsilon}{2}(\frac{1}{2}-\rho)^2,$$
$$V(s)=-\frac{m_1}{1+\rho}+\frac{m_2}{\rho}-\frac{\epsilon}{2}(\frac{1}{2}+\rho)^2,$$
$$V(r)=-\frac{m_1}{1-\rho}-\frac{m_2}{2-\rho}-\frac{\epsilon}{2}(\frac{3}{2}-\rho)^2.$$
and
\begin{equation}\label{Vs-Vt}V(s)-V(t)=\rho(\frac{2m_1}{1-\rho^2}-\epsilon).\end{equation}
\begin{equation}\label{Vr-Vt}V(r)-V(t)=(1-\rho)(\frac{2m_2}{\rho(2-\rho)}-\epsilon).\end{equation}
Since $0<\rho\leq \frac{1}{2}$, we get
$$1<\frac{1}{1-\rho^2}\leq \frac{4}{3},$$
$$\frac{1}{\rho(2-\rho)}\geq \frac{4}{3},$$
and 
\begin{equation}\label{compare}
2m_1-\epsilon<\frac{2m_1}{1-\rho^2}-\epsilon\leq \frac{8m_1}{3}-\epsilon.\end{equation}
\begin{equation}\label{compare2}
\frac{2m_2}{\rho(2-\rho)}-\epsilon\geq \frac{8m_2}{3}-\epsilon.\end{equation}

By observing (\ref{Vs-Vt}), (\ref{Vr-Vt}), (\ref{compare}) and (\ref{compare2}), we have the following estimate. 
If $m_1\geq \frac{\epsilon}{2}$, then by (\ref{compare}), $V(t)<V(s)$ for all $0\leq t<\frac{1}{2}$.
As a result 
$$V(l_1)=max\{V(t), 0\leq t<\frac{1}{2}\}<V(s)\leq V(l_2).$$
If $m_2\geq \frac{3\epsilon}{8}$, then by (\ref{compare2}), $V(t)<V(r)$ for all $0\leq t<\frac{1}{2}$.
As a result 
$$V(l_1)=max\{V(t), 0\leq t<\frac{1}{2}\}<V(r)\leq V(l_3).$$
Together with (\ref{V2<V3}), we get the result that
when $$m_1\geq\frac{\epsilon}{2}, m_2\geq\frac{3\epsilon}{8}$$ we have $$V(l_1)<V(l_3)<V(l_2).$$
and $$H_{\epsilon}(L_1)<H_{\epsilon}(L_3)<H_{\epsilon}(L_2).$$
If $m_1<\frac{3\epsilon}{8}$, by (\ref{Vs-Vt}) and (\ref{compare}),
\begin{equation}\label{Vt>Vs}V(t)>V(s),0\leq t<\frac{1}{2} .\end{equation}
 By (\ref{D1U}) and (\ref{compare}),
$$\frac{\partial V}{\partial q_1}\bigg|_{(1,0)}=\frac{4m_1}{9}+4m_2-\epsilon\leq 4m_2-\frac{5\epsilon}{6}.$$
We observe that when $m_2<\frac{5\epsilon}{24}$, 
$$\frac{\partial U}{\partial q_1}\bigg|_{(1,0)}<0,$$
Since when $q_1\rightarrow\pm\frac{1}{2}$ or $q_1\rightarrow\pm\infty$, $V(q)$ all go to $-\infty$, and $V(q_1,0)$ is convex on $(-\infty, -\frac{1}{2})$, $(-\frac{1}{2},\frac{1}{2})$ and $(\frac{1}{2},\infty)$ separately by $(\ref{partial U q1q1 on x-axis})$, $$\frac{1}{2}<l_2<1.$$ 
as a result,
$$V(l_2)=max\{V(s), \frac{1}{2}<s\leq 1\}<V(t)\leq V(l_1).$$
Together with (\ref{V2<V3}), we get the result that when $$m_1< \frac{3\epsilon}{8}, m_2\leq \frac{5\epsilon}{24}$$ we have
$$V(l_3)<V(l_2)<V(l_1).$$
and $$H_{\epsilon}(L_3)<H_{\epsilon}(L_2)<H_{\epsilon}(L_1)$$
\end{proof}

\begin{remark}
When $m_1<m_2$, we have similar results as in lemma \ref{lem2}.
\end{remark}

We consider the energy surface \begin{equation}\label{Sigma}\Sigma_{c,\epsilon}:=H_{\epsilon}^{-1}(c)\subset T^*(\mathbb{R}^2\backslash\{(\pm\frac{1}{2},0)\}).\end{equation}
Define the Hill's region to be the shadow of the energy hypersurface
$$\mathcal{R}_{c,\epsilon}:=\pi(\Sigma_{\epsilon})=\{q\in\mathbb{R}^2\backslash\{(\pm\frac{1}{2},0)\}:V_{\epsilon}(q)\le c \}.$$
Abbreviate $c_0=min\{H_{\epsilon}(L_1), H_{\epsilon}(L_2), H_{\epsilon}(L_3)\}$. For $c<c_0$, the Hill's region consists of three connected components, two bounded around the two centers $e$ and $m$ seperately and one unbounded 
$$\mathcal{R}_{c,\epsilon}^e\cup\mathcal{R}_{c,\epsilon}^m\cup\mathcal{R}_{c,\epsilon}^u.$$
The energy hypersurface itself decomposes into three connected components $$\Sigma_{c,\epsilon}=\Sigma_{c,\epsilon}^e\cup\Sigma_{c,\epsilon}^m\cup\Sigma_{c,\epsilon}^u$$
satisfying $\pi(\Sigma_{c,\epsilon}^e)=\mathcal{R}_{c,\epsilon}^e$, $\pi(\Sigma_{c,\epsilon}^m)=\mathcal{R}_{c,\epsilon}^m$ and $\pi(\Sigma_{c,\epsilon}^u)=\mathcal{R}_{c,\epsilon}^u$.


While for Lagrange problem with $m_1>0, m_2<0,\epsilon>0$. When $q_2=0$, if $q_1\rightarrow -\frac{1}{2}$ or $q_1\rightarrow \pm \infty$, $-U_{\epsilon}$ all go to $-\infty$, If $q_1\rightarrow \frac{1}{2}$, $-U_{\epsilon}$ goes to $+\infty$,  As a result, there is only one maxima of $U$ restricted in the x-axis $l=(\iota,0)$ and $l<-\frac{1}{2}$. When the energy is less then the first critical value, i.e. $c<-U(l)$, the orbits are in the bounded region near the body with mass $m_1$. 

For Euler problem with $m_1>0, m_2>0$, we get the following lemma.
\begin{lemma}\label{lem3}
For Euler problem, assume that $m_1>0, m_2>0$, by a direct computation, we know that there is only one critical point $l$ in the x-axis in the region $-\frac{1}{2}<\iota<\frac{1}{2}$ and $L$ is the maxima of $H_0$ restricted to the x-axis, where $L=(0, l)$, $l=(\iota,0)$ and $$\iota=\frac{1}{2}-\frac{1}{\sqrt{\frac{m_1}{m_2}}+1}.$$The critical value is $$H_0(L)=-(\sqrt{m_1}+\sqrt{m_2})^2.$$
\end{lemma}
\begin{proof}
By (\ref{D1U}) and (\ref{D2U}), the critical point satisfies
\begin{equation}\label{D1U of Euler}
\frac{m_1(q_1+\frac{1}{2})}{\big((q_1+\frac{1}{2})^2+q_2^2\big)^{\frac{3}{2}}}+\frac{m_2(q_1-\frac{1}{2})}{\big((q_1-\frac{1}{2})^2+q_2^2\big)^{\frac{3}{2}}}=0,
\end{equation}
\begin{equation}
\frac{m_1q_2}{\big((q_1+\frac{1}{2})^2+q_2^2\big)^{\frac{3}{2}}}+\frac{m_2 q_2}{\big((q_1-\frac{1}{2})^2+q_2^2\big)^{\frac{3}{2}}}=0.
\end{equation}

When $q_2\neq 0$, similar as (\ref{equ of Lj}), they are equivalent to
\begin{equation}
\frac{m_2}{\big((q_1-\frac{1}{2})^2+q_2^2\big)^{\frac{3}{2}}}=\frac{m_1}{\big((q_1+\frac{1}{2})^2+q_2^2\big)^{\frac{3}{2}}}=0,
\end{equation}
which has no solution.

When $q_2=0$, the critical point satisfies
\begin{equation}\label{crit of Euler pos}
\frac{m_1(q_1+\frac{1}{2})}{|q_1+\frac{1}{2}|^3}+\frac{m_2(q_1-\frac{1}{2})}{|q_1-\frac{1}{2}|^3}=0,
\end{equation}
In the region of $q_1>\frac{1}{2}$,  (\ref{crit of Euler pos}) is equivalent to 
$$\frac{m_1}{|q_1+\frac{1}{2}|^2}+\frac{m_2}{|q_1-\frac{1}{2}|^2}=0,$$ which has no solution.
In the region of $-\frac{1}{2}<q_1<\frac{1}{2}$,  (\ref{crit of Euler pos}) is equivalent to 
$$\frac{m_1}{|q_1+\frac{1}{2}|^2}-\frac{m_2}{|q_1-\frac{1}{2}|^2}=0,$$ which has only one solution 
$$\iota=\frac{1}{2}-\frac{1}{\sqrt{\frac{m_1}{m_2}}+1}.$$
The critical value in this critical point $l=(\iota,0)$ is
$$U_0(l)=(\sqrt{m_1}+\sqrt{m_2})^2.$$
and 
$$H_0(L)=-(\sqrt{m_1}+\sqrt{m_2})^2,$$
by (\ref{H=-U}).
In the region of $q_1<-\frac{1}{2}$, (\ref{crit of Euler pos}) is equivalent to 
$$-\frac{m_1}{|q_1+\frac{1}{2}|^2}-\frac{m_2}{|q_1-\frac{1}{2}|^2}=0,$$ which also has no solution.

In conclusion, there is only one critical point $L$ under the condition of this lemma.

Since $H_0\rightarrow0$ when $q_1\rightarrow\pm \infty$, $H_0\rightarrow-\infty$ when $q_1\rightarrow\pm\frac{1}{2}$, the critical point $L$ must be the maximum of $H_0$ restricted to the x-axis, i.e.
$$\frac{\partial^2 V}{\partial q_1^2}(l)<0.$$
 Therefore, $L$ is the maxima of $H$ restricted to the x-axis.
\end{proof}

The following estimate for the critical point $L$ of the Euler problem with $m_1>0, m_2>0$ is useful in my paper.
\begin{equation}\label{estimate 1}
H_0(L)<-(m_1+m_2).
\end{equation}

While for Euler problem with $m_1>0, m_2<0, m_1>|m_2|$, we have the following lemma.
\begin{lemma}\label{lem4}
The Euler problem with $m_1>0, m_2<0, m_1>|m_2|$ has only one critical point $l=(\iota,0)$ in the x-axis in the region $\iota>\frac{1}{2}$ where 
$$\iota=\frac{1}{2}+\frac{1}{\sqrt{-\frac{m_1}{m_2}}}$$ and $L$ is the minimmu of $H_0$ restricted to the x-axis. The critical value is 
$$H_0(L)=-(\sqrt{m_1}+\sqrt{-m_2})^2.$$ 
\end{lemma}

\begin{proof}
By (\ref{D1U}) and (\ref{D2U}), the critical point satisfies
\begin{equation}
\frac{m_1(q_1+\frac{1}{2})}{\big((q_1+\frac{1}{2})^2+q_2^2\big)^{\frac{3}{2}}}+\frac{m_2(q_1-\frac{1}{2})}{\big((q_1-\frac{1}{2})^2+q_2^2\big)^{\frac{3}{2}}}=0,
\end{equation}
\begin{equation}
\frac{m_1q_2}{\big((q_1+\frac{1}{2})^2+q_2^2\big)^{\frac{3}{2}}}+\frac{m_2 q_2}{\big((q_1-\frac{1}{2})^2+q_2^2\big)^{\frac{3}{2}}}=0.
\end{equation}

When $q_2\neq 0$, similar as (\ref{equ of Lj}), they are equivalent to
\begin{equation}
\frac{m_2}{\big((q_1-\frac{1}{2})^2+q_2^2\big)^{\frac{3}{2}}}=\frac{m_1}{\big((q_1+\frac{1}{2})^2+q_2^2\big)^{\frac{3}{2}}}=0,
\end{equation}
which has no solution.

When $q_2=0$, the critical point satisfies (\ref{crit of Euler pos})
In the region of $q_1>\frac{1}{2}$,  (\ref{crit of Euler pos}) is equivalent to 
$$\frac{m_1}{|q_1+\frac{1}{2}|^2}+\frac{m_2}{|q_1-\frac{1}{2}|^2}=0,$$ which has only one solution 
$$\iota=\frac{1}{2}+\frac{1}{\sqrt{-\frac{m_1}{m_2}}-1}.$$
The critical value in this critical point $l=(\iota,0)$ is
$$U_0(l)=(\sqrt{m_1}+\sqrt{-m_2})^2$$
and 
$$H_0(L)=-(\sqrt{m_1}+\sqrt{-m_2})^2$$
by (\ref{H=-U}).

In the region of $-\frac{1}{2}<q_1<\frac{1}{2}$,  (\ref{crit of Euler pos}) is equivalent to 
$$\frac{m_1}{|q_1+\frac{1}{2}|^2}-\frac{m_2}{|q_1-\frac{1}{2}|^2}=0,$$ which has no solution.
In the region of $q_1<-\frac{1}{2}$, (\ref{crit of Euler pos}) is equivalent to 
$$-\frac{m_1}{|q_1+\frac{1}{2}|^2}-\frac{m_2}{|q_1-\frac{1}{2}|^2}=0,$$ which also has no solution.

In conclusion, there is only one critical point $L$ under the condition of this lemma.

Since $H_0\rightarrow0$ when $q_1\rightarrow\pm \infty$, $H_0\rightarrow-\infty$ when $q_1\rightarrow-\frac{1}{2}$, $H_0\rightarrow+\infty$ when $q_1\rightarrow \frac{1}{2}$, the critical point $L$ must be the minimum of $H_0$ restricted to the x-axis, i.e.
$$\frac{\partial^2 V}{\partial q_1^2}(l)>0.$$
 Therefore, $L$ is the maxima of $H_0$ restricted to the x-axis.
\end{proof}

\section{regularization of the Lagrange problem}

As the energy hypersurface even over the bounded component is never compact due to collisions with the two fixed centers $e$ and $m$. In this section we express the Lagrange problem in the elliptic coordinates, and then regularize the collisions. After regularization, the Lagrange problem can be seperated as well. This shows then that the Lagrange problem is completely integrable.

The map from elliptic coordinates $(\mu, \nu)$ to the initial coordinates $(q_1, q_2)$ is a two-to-one covering
$$l:(\mathbb{R}\times S^1)\backslash\{e,m\}\rightarrow\mathbb{R}^2\backslash\{e,m\}, 
(\mu,\nu)\mapsto(q_1,q_2)$$
where
\begin{equation}\label{q1,q2}
\left\{
\begin{aligned}
	q_1&=\frac{1}{2}\cosh\mu\cdot\cos\nu,\\
	q_2&=\frac{1}{2}\sinh\mu\cdot\sin\nu.
\end{aligned}
\right.
\end{equation}
$S^1=\mathbb{R}/2\pi\mathbb{Z}$. In the elliptic coordinate, $e=(-\frac{1}{2},0)$, $m=(\frac{1}{2},0)$. In the original coordinate, $e=(0,\pi)$, $m=(0, 0)$.
The map $l$ lifts to an exact two-to-one symplectic covering
$$\begin{aligned}L : &T^*(\mathbb{R}\times S^1)\backslash\{e,m\}\rightarrow T^*(\mathbb{R}^2\backslash\{e,m\}), \\
&(\mu,\nu, p_{\mu},p_{\nu})\mapsto (q_1,q_2,p_1,p_2).\end{aligned}$$

The Jacobi matrix from $(q_1,q_2)$ to $(\mu,\nu)$ is
\begin{equation}
D_1:=\frac{\partial(q_1,q_2)}{\partial(\mu,\nu)}=\left[\begin{matrix}
\frac{\partial q_1}{\partial \mu} & \frac{\partial q_1}{\partial \nu} \\
\frac{\partial q_2}{\partial \mu} & \frac{\partial q_2}{\partial \nu} 
\end{matrix}
\right]\\
=\frac{1}{2}\left[
\begin{matrix} 
	\sinh\mu\cdot\cos\nu & -\cosh\nu\cdot\cos\nu\\ \cosh\mu\cdot\sin\nu&\sinh\mu\cdot\cos\nu\end{matrix}
\right],
\end{equation}
and its determinant is
$$\det{D_1}=\frac{1}{4}(\cosh^2\mu-\cos^2\nu).$$
and
$$D_1^{-T}=\frac{D_1}{\det D_1},$$
then
\begin{equation}\label{p_1 to p_mu}\begin{aligned}
\left(\begin{matrix}p_1\\p_2\end{matrix}\right)
&=D_1^{-T}\left(\begin{matrix}p_{\mu}\\p_{\nu}\end{matrix}\right)\\
&=\frac{2}{\cosh^2\mu-\cos^2\nu}\left(\begin{matrix}\sinh\mu\cdot\cos\nu \cdot p_{\mu}-\cosh\mu\cdot\sin\nu \cdot p_{\nu} \\ \cosh\mu\cdot\sin\nu \cdot p_{\mu}+\sinh\mu\cdot\cos\nu \cdot p_{\nu} \end{matrix}\right).\end{aligned}\end{equation}

So we get the symplectic transformation $$L : (p_{\mu},p_{\nu},\mu,\nu)\mapsto (p_1,p_2,q_1,q_2)$$ and the Hamiltonian is transformed into
$$H_{\epsilon}(\mu,\nu,P_\mu,P_\nu)=T(\mu, \nu, P_\mu,P_\nu)-U_{\epsilon}(\mu,\nu)$$
where

$$T(\mu,\nu,P_\mu,P_\nu)=\frac{4}{\cosh^2\mu-\cos^2\nu}\big(\frac{1}{2}P_\mu^2+\frac{1}{2}P_\nu^2\big).$$

$$\begin{aligned}U_{\epsilon}(\mu,\nu,P_\mu,P_\nu)=&\frac{1}{\cosh^2\mu-\cos^2\nu}(\frac{\epsilon}{8}\cosh^4\mu-\frac{\epsilon}{8}\cos^4\nu-\frac{\epsilon}{8}\cosh^2\mu+\frac{\epsilon}{8}\cos^2\nu\\&+2M_1\cosh\mu-2M_2\cos\nu)\end{aligned}$$
and
$$M_1=m_1+m_2, M_2=m_1-m_2.$$

We now define for $(\mu,\nu)\in T^*(\mathbb{R}\times [0,2\pi)\backslash\{(0,0),(0,\pi)\})$
\begin{equation}\label{H to K}K_{c,\epsilon}
=\frac{1}{4}(\cosh^2\mu-\cos^2\nu)(H_{\epsilon}-c).\end{equation}

Explicitly this becomes
\begin{equation}\label{K=K1+K2}K_{c,\epsilon}=K_{c,\epsilon}^1+K_{c,\epsilon}^2.\end{equation}
for 
\begin{equation}\label{K1K2}\begin{aligned}K_{c,\epsilon}^1
&=\frac{1}{2}p_{\mu}^2-\frac{\epsilon}{32}\cosh^4\mu+\frac{\epsilon}{32}\cosh^2\mu-\frac{c}{4}(\cosh^2\mu-1)-\frac{M_1}{2}\cosh\mu\\:&=\frac{1}{2}p_{\mu}^2+W_{c,\epsilon}^1(\mu).\\
K_{c,\epsilon}^1
&=\frac{1}{2}p_{\nu}^2+\frac{\epsilon}{32}\cos^4\nu-\frac{\epsilon}{32}\cos^2\nu+\frac{c}{4}(\cos^2\nu-1)+\frac{M_2}{2}\cos\nu\\:&=\frac{1}{2}p_{\nu}^2+W_{c,\epsilon}^2(\nu).\end{aligned}\end{equation}

From (\ref{K=K1+K2}), we found that $K_{\epsilon}$ makes sense for every $(\mu,\nu,p_{\mu},p_{\nu})\in T^*(\mathbb{R}\times S^1)$, so we can add the fiber over $e=(0,0)$ and $m=(0,\pi)$ and interpret $K_{c,\epsilon}$ as a smooth function
$$K_{c,\epsilon} : T^*(\mathbb{R}\times S^1)\rightarrow\mathbb{R}$$defined by (\ref{K=K1+K2}).
The original definiton (\ref{H to K}) leads to the equality
$$K_{c,\epsilon}|_{T^*(\mathbb{R}\times S^1)}=R\cdot(L^*H_{\epsilon}-c)$$for$$R : T^*(\mathbb{R}\times S^1)\rightarrow\mathbb{R}, (\mu,\nu)\mapsto \frac{1}{4}(\cosh^2\mu-\cos^2\nu).$$
Then we have $$L^{-1}(\Sigma_{c,\epsilon})=K_{c,\epsilon}|^{-1}_{T^*(\mathbb{R}\times S^1\backslash\{e,m\})}(0)$$
Since L is a symplectic covering, from (\ref{H to K}), it holds that
$$dK_{c,\epsilon}=\frac{1}{4}d(\cosh^2\mu-\cos^2\nu)\cdot(H_\epsilon-c)+\frac{1}{4}(\cosh^2\mu-\cos^2\nu)dH_{\epsilon}.$$

Together with (\ref{Sigma}), we get
$$dK_{c,\epsilon}=\frac{1}{4}(\cosh^2\mu-\cos^2\nu)dH_{\epsilon}$$
and then
$$X_{K_{c,\epsilon}}|_{L^{-1}(\Sigma_{c,\epsilon})}=R\cdot L^*X_{H_{\epsilon}}|_{\Sigma_{c,\epsilon}}.$$
That means that up to reparametrization the restriction of the flow of $X_{H_{\epsilon}}$ to the energy hypersurface $\Sigma_{c,\epsilon}$ can be identified with the flow of $X_{X_{c,\epsilon}}$ restricted to the preimage of $\Sigma_{c,\epsilon}$ under $L$.

We now set
\begin{equation}\label{bar Sigma}\overline{\Sigma}_{c,\epsilon} :=K_{c,\epsilon}^{-1}(0)\subset T^*(\mathbb{R}\times S^1),\end{equation}
which contains $L^{-1}(\Sigma_{\epsilon})$ as a dense and open subset. The complement
$$\overline \Sigma_{c,\epsilon}\backslash L^{-1}(\Sigma_{c,\epsilon})=\overline \Sigma_{c,\epsilon}\cap T^*_{\{e, m\}}(\mathbb{R}\times S^1)$$
contains precisely the collisions where after time change the vector field now extends smoothly.
For $c<c_0$, the regularized energy hypersuface as the unregularized one decomposes into two bounded connected components and an unbounded part $$\overline{\Sigma}_{c,\epsilon}=\overline{\Sigma}_{c,\epsilon}^e\cup\overline{\Sigma}_{c,\epsilon}^m\cup \overline{\Sigma}_{c,\epsilon}^u.$$
The bounded part contains the collisions $\overline{\Sigma}_{\epsilon}\cap T^*_{\{e,m\}}(\mathbb{R}\times S^1)$. The unbounded part actually is diffeomorphic to two copies of the unbounded componet $\Sigma_{c,\epsilon}^u$ via the two-to-one map $L$. In fact we just have $$\overline{\Sigma}_{\epsilon}^u=L^{-1}(\Sigma_{\epsilon}^u).$$

From (\ref{K=K1+K2}) we see that $K_{c,\epsilon}$ can be written as the sum of two Poisson commuting Hamiltonians. In particular, we see that the Lagrange problem is completely integrable. Due to the separability of the generalized problem we can slice the energy hypersurface. For that purpose we abbreviate 
$$S_{\kappa,c,\epsilon}^1 :=(K_{c,\epsilon}^1)^{-1}(\kappa)\cap T^*(\mathbb{R}\times S^1),$$
$$S_{\kappa,c,\epsilon}^2 :=(K_{c,\epsilon}^2)^{-1}(\kappa)\cap T^*(\mathbb{R}\times S^1).$$

By (\ref{K=K1+K2}) and (\ref{bar Sigma}), when $K_{c,\epsilon}^1=-\kappa$, we have $K_{c,\epsilon}^2=\kappa$. We now analyze the shapes of the potential energy $W_{c,\epsilon}^1(\mu)$ and $W_{c,\epsilon}^2(\nu)$ of the separated systems  to determine the dynamical behaviors of the problem.
Note
\begin{equation}\label{f(x)}f(x)=-\frac{\epsilon}{8}x^3-\frac{c}{2}x+\frac{\epsilon}{16}x-\frac{M_1}{2}, x=\cosh\mu,\end{equation}
\begin{equation}\label{g(y)}g(y)=\frac{\epsilon}{8}y^3+\frac{c}{2}y-\frac{\epsilon}{16}y+\frac{M_2}{2}, y=\cos\nu.\end{equation}
By (\ref{K1K2}), 
$$\begin{aligned}&\frac{\partial}{\partial\mu}W_{c,\epsilon}^1(\mu)=f(x)\cdot\sinh\mu,\\
&\frac{\partial^2}{\partial\mu^2}W_{c,\epsilon}^1(\mu)=f'(x)\cdot\sinh^2\mu+f(x)\cdot\cosh\mu.\\
&\frac{\partial}{\partial\nu}W_{c,\epsilon}^2(\nu)=g(y)\cdot(-\sin\nu),\\
&\frac{\partial^2}{\partial\nu^2}W_{c,\epsilon}^2(\nu)=g'(y)\cdot\sin^2\nu+g(y)\cdot(-\cos\nu).\end{aligned}$$
then
\begin{equation}\label{f(1)}\frac{\partial^2}{\partial\mu^2}W_{c,\epsilon}^1(0)=f(1)=-\frac{1}{2}\big(\frac{\epsilon}{8}+c+M_1\big),\end{equation}
\begin{equation}\label{g(1)}\frac{\partial^2}{\partial\nu^2}W_{c,\epsilon}^2(0)=-g(1)=-\frac{1}{2}\big(\frac{\epsilon}{8}+c+M_2\big),\end{equation}
\begin{equation}\label{g(-1)}\frac{\partial^2}{\partial\nu^2}W_{c,\epsilon}^2(\pi)=g(-1)=-\frac{1}{2}\big(\frac{\epsilon}{8}+c-M_2).\end{equation}
We see that
\begin{equation}\label{k m}\frac{\partial^2}{\partial\mu^2}W_{c,\epsilon}^1(0)>0, \frac{\partial^2}{\partial\nu^2}W_{c,\epsilon}^2(0)>0.\end{equation}
\begin{equation}\label{k e}\frac{\partial^2}{\partial\mu^2}W_{c,\epsilon}^1(0)>0, \frac{\partial^2}{\partial\nu^2}W_{c,\epsilon}^2(\pi)>0\end{equation}
 hold if and only if  
\begin{equation}\label{condition of torus}\frac{\epsilon}{8}+c+M_1<0.\end{equation}
(\ref{k e}) implies that the graph of $W_{c,\epsilon}^1(\mu)$ and $W_{c,\epsilon}^2(\nu)$ are like cups near the point $e=(0, \pi)$ and (\ref{k m}) implies this property near the point $m=(0, 0)$.

In the following, we assume the energy $c<c_0$ and also the condition (\ref{condition of torus}) holds.

From (\ref{f(x)}), we know $f(0)=-\frac{M_1}{2}<0$, $f(x)\rightarrow +\infty$ when $x\rightarrow -\infty$ and $f(x)\rightarrow -\infty$ when $x\rightarrow +\infty$. Together with (\ref{f(1)}), we know that $f(x)$ has only one root $x_0$ in $(1,+\infty)$. As a result, $W_{c,\epsilon}^1(\mu)$ has just three roots $0$, $\pm \mu_0$ for $x_0=\cosh \mu_0$. 

From (\ref{g(y)}), we know $g(0)=\frac{M_2}{2}\geq0$, $g(y)\rightarrow -\infty$ when $x\rightarrow -\infty$ and $g(y)\rightarrow +\infty$ when $x\rightarrow +\infty$. Together with (\ref{g(1)}), we know that $g(y)$ has only one root $y_0$ in $[0,1)$. As a result, $W_{c,\epsilon}^2(\nu)$ has just three roots $0$, $\pm y_0$ for $y_0=\cos \nu_0$.

Now we show that \begin{equation}\label{circle 1}W_{c,\epsilon}^2(\nu_0)>-W_{c,\epsilon}^1(0)\end{equation} and \begin{equation}\label{circle 2}W_{c,\epsilon}^2(0)>-W_{c,\epsilon}^1(\mu_0)\end{equation} when $c<c_0$ by contradiction. 

Assume $W_{c,\epsilon}^2(\nu_0)\leq-W_{c,\epsilon}^1(0)$.
If $c$ is small enough, then the function values of $-W_{c,\epsilon}^1$ and $W_{c,\epsilon}^2$ intersect at a very small interval compared to the whole function value, so we have $W_{c,\epsilon}^2(\nu_0)>-W_{c,\epsilon}^1(0)$ now. By continuity of $-W_{c,\epsilon}^1$ and $W_{c,\epsilon}^2$ with respect to $c$, there should exits a $c_x$, such that $W_{c,\epsilon}^2(\nu_0)=-W_{c,\epsilon}^1(0)$. But this $c_x$ must be a Lagrange value, since $W_{c,\epsilon}^1(\mu)$ and $W_{c,\epsilon}^2(\nu)$ each just has three roots, these roots can only correspond to Lagrange points, namely in the elliptic coordinate if $c=H(L_1)$, then $L_1=(\pm \mu_0,0)$, if $c=H(L_2)$, then $L_2=(0, \pm\nu_0)$, if $c=H(L_3)$, then $L_3=(\pi, \pm\nu_0)$.

Because of the analysis above, under the condition of (\ref{condition of torus}) and $c<c_0$, we get 
$$S^1_{\kappa,c,\epsilon}=S^{1,e}_{\kappa,c,\epsilon}\cup S^{1,m}_{\kappa,c,\epsilon}\cup S^{1,u}_{\kappa,c,\epsilon}$$
\begin{equation}\label{slicing of e}
\overline{\Sigma}^e_{c,\epsilon}=\bigcup_{\kappa\in\big[-\frac{M_2}{2}, \frac{M_1}{2}\big]}S_{-\kappa,c,\epsilon}^{1,e}\times S_{\kappa,c,\epsilon}^{2,e}\end{equation}
\begin{equation}\label{slicing of m}
\overline{\Sigma}^m_{c,\epsilon}=\bigcup_{\kappa\in\big[\frac{M_2}{2}, \frac{M_1}{2}\big]}S_{-\kappa,c,\epsilon}^{1,m}\times S_{\kappa,c,\epsilon}^{2,m}\end{equation}
where $S^{1,e}_{\kappa,c,\epsilon}$ and $S^{1,m}_{\kappa,c,\epsilon}$ are the bounded components around $e$ and $m$ separately, $S^{1,u}_{\kappa,c,\epsilon}$ is the unbounded part.
$S_{-\kappa, c,\epsilon}^{1,e}\times S_{\kappa, c,\epsilon}^{2,e}$ is an Arnold-Liouville torus expected for a completely integrable system for $-\frac{M_2}{2}<\kappa<\frac{M_1}{2}$ and the Arnold-Liouville torus degenerates to a circle for $\kappa=-\frac{M_2}{2}$ or $\kappa=-\frac{M_1}{2}$. Similarly, $S_{-\kappa,c,\epsilon}^{1,m}\times S_{\kappa,c,\epsilon}^{2,m}$ is an Arnold-Liouville torus for $\frac{M_2}{2}<\kappa<\frac{M_1}{2}$ and the Arnold-Liouville torus degenerates to a circle for $\kappa=\frac{M_2}{2}$ or $\kappa=\frac{M_1}{2}$. 

Define $c_{crit}=-\frac{\epsilon}{8}-M_1$, we try to compare $c_{crit
}$ with $c_0$ and get the following Lemma {\ref{lem6}} and Lemma {\ref{lem7}}.
\begin{lemma}\label{lem6}
For the Lagrange problem with $m_1=m_2=m$ and $m\geq\frac{\epsilon}{2}\geq0$, we have $c_0<c_{crit}$.
\end{lemma}
\begin{proof}
Now we compute the first critical point $l=(\iota,0)$ of the Lagrange problem.
By (\ref{D1U}), 
$$\frac{m}{(\iota+\frac{1}{2})^2}-\frac{m}{(\iota-\frac{1}{2})^2}-\epsilon\iota=0,$$
this implies 
$$\iota\cdot(\frac{2m}{(\iota+\frac{1}{2})^2(\iota-\frac{1}{2})^2}+\epsilon)=0$$
since $m>0, \epsilon>0$, we have
$$\iota=0.$$
Then the first critical value is 
$$V(l)=V((0,0))=-4m.$$
By lemma \ref{lem2}, $V(l)=c_0=min\{H_{\epsilon}(L_1), H_{\epsilon}(L_2), H_{\epsilon}(L_3)\}$.
Since $m\geq\frac{\epsilon}{2}$, we have
$$V(l)=-4m=-M_1-M_1<-M_1-\epsilon<-M_1-\frac{\epsilon}{8}.$$
\end{proof}

This lemma is a special case of the following lemma, which also discusses the general case when $m_1\neq m_2$. 
\begin{lemma}\label{lem7}
For Lagrange problem with $m_1>0, m_2>0, \epsilon>0$, if $m_2\leq m_1\leq 9m_2$, $m_1\geq\frac{\epsilon}{2}$, $m_2\geq\frac{3\epsilon}{8}$, then we have $c_0<c_{crit}$.
\end{lemma}
\begin{proof}
Note the first critical point of the Lagrange problem as $l=(\iota,0)$. Since $m_1\geq m_2$, we have $0<l<\frac{1}{2}$
.
By (\ref{D1U}), 
\begin{equation}\label{l1}\frac{m_1}{(\iota+\frac{1}{2})^2}-\frac{m_2}{(\iota-\frac{1}{2})^2}-\epsilon\iota=0.\end{equation}
The first critical value is 
$$V(l)=-\frac{m_1}{\iota+\frac{1}{2}}+\frac{m_2}{\iota-\frac{1}{2}}-\frac{\epsilon}{2}\iota^2.$$
By (\ref{l1}), we have
\begin{equation}\label{V(iota)}\begin{aligned}
V(\iota)
=&\big(\frac{1}{4(\iota+\frac{1}{2})^2}-\frac{3}{2(\iota+\frac{1}{2})}\big)m_1+\big(\frac{1}{4(\iota-\frac{1}{2})^2}+\frac{3}{2(\iota-\frac{1}{2})}\big)m_2\\
:=&Q_1m_1+Q_2m_2.
\end{aligned}\end{equation}

Set $$F_{\epsilon} :(0,\frac{1}{2})\rightarrow \mathbb{R},\iota\mapsto\frac{m_1}{(\iota+\frac{1}{2})^2}-\frac{m_2}{(\iota-\frac{1}{2})^2}-\epsilon\iota,$$
then $F_{\epsilon}$ is a decrease function with respect to $\iota$.
By (\ref{l1}), when $\epsilon=0$, the first cirtical point is $l_0=(\iota_0,0),$

$$\iota_0=\frac{1}{2}-\frac{1}{\sqrt{\frac{m_1}{m_2}}+1.}$$
i.e. $F_{0}(\iota_0)=0$.
Set the first critical point of the Lagrange problem as $l_1=(\iota_1,0)$, i.e.$F_{\epsilon}(\iota_1)=0$. Since for $\epsilon\geq 0$
$F_{\epsilon}(\iota_0)\leq F_{0}(\iota_0)=0$. By the decreasing of the funcction $F_{\epsilon}$, we get $\iota_1\leq\iota_0$.

If $1\leq\frac{m_1}{m_2}\leq9$, we get $0\leq \iota_1\leq \iota_0\leq \frac{1}{4}$.
By a direct computation, we know that in (\ref{V(iota)}), $$-2\leq Q_1\leq-\frac{14}{9}, -\frac{9}{4}\leq Q_2\leq -2.$$
As a result,
$$V(l_1)\leq -\frac{14}{9}m_1-2m_2\leq-\frac{14}{9}M_1.$$
By lemma \ref{lem2}, $V(l_1)=c_0=min\{H_{\epsilon}(L_1), H_{\epsilon}(L_2), H_{\epsilon}(L_3)\}$ under the conditions $m_1\geq\frac{\epsilon}{2}$, $m_2\geq\frac{3\epsilon}{8}$. $$c_0\leq -\frac{14}{9}M_1=-M_1-\frac{5}{9}M_1<-M_1-\frac{\epsilon}{8}.$$

\end{proof}

\begin{remark}
Consider the Lagrange problem with $m_2=0$, it has two critical points,  the first one in the x-axis with $x<0.5$ and the second one with $x>0$. We can find many counter examples such that $c_0>c_{crit}$. For example, if $m_1=80$, $m_2=0$ and $\epsilon=8$, then $c_0=-65$, $c_{crit}=-81$. 

By computer experiment, all the counter examples we get has the second critical point on $x>0.5$. So we can't get a counter example for $m_2\neq 0$ by continuity. 

Moreover, by computer experiment, we can't find a counter example for the Lagrange problem with $m_2\neq 0$ for a large range of $m_1>0$, $m_2>0$ and $\epsilon>0$. So we have reason to guess that $c_0<c_{crit}$ actually hold for the Lagrange problem with $m_1>0$, $m_2>0$ and $\epsilon>0$. This can be recognized as an open problem.
\end{remark}

\section{The moment map}

Now we define a torus action on the regularized moduli space $\overline{\Sigma}_{c,\epsilon}^{e}$. In order to do that we first need the periods. The set $S_{c,\kappa,\epsilon}^1$ is diffeomorphic to a circle, which coincides with the periodic orbit of the Hamiltonian $K_{c,\epsilon}^1$ of energy $-\kappa$.

Set
\begin{equation}\label{x,y}
\left\{
\begin{aligned}
	x&=\cosh\mu,\\
	y&=\cos\nu.
\end{aligned}
\right.
\end{equation}

Then the periods are as follows.
\begin{equation}
\begin{aligned}
\frac{\tau^1_{c,\epsilon}(-\kappa)}{4}&=\int_0^{\frac{\tau^1_{c,\epsilon}(-\kappa)}{4}}dt\\
&=\int_0^{u_0}\frac{d\mu}{\sqrt{2(-\kappa-W_{c,\epsilon}^1)}}\\
&=\int_{1}^{x_0}\frac{dx}{\sqrt{2(x^2-1)\big(-\kappa+\frac{\epsilon}{32}x^4+\frac{c}{4}(x^2-1)-\frac{\epsilon}{32}x^2+\frac{M_1}{2} x\big)}},
\end{aligned}
\end{equation}
\begin{equation}
\begin{aligned}
\frac{\tau^2_{c,\epsilon}(\kappa)}{4}&=\int_0^{\frac{\tau^2_{c,\epsilon}(\kappa)}{4}}dt\\
&=\int_0^{\nu_0}\frac{d\nu}{\sqrt{2(\kappa-W_{c,\epsilon}^2)}}\\
&=\int_{y_0}^{1}\frac{dy}{\sqrt{2(1-y^2)\big(\kappa-\frac{\epsilon}{32}y^4-\frac{c}{4}(y^2-1)+\frac{\epsilon}{32}y^2-\frac{M_2}{2} y\big)}}.\\\end{aligned}
\end{equation}

Denote by $\Phi^t_{K_{c,\epsilon}^1}$ the flow of the Hamiltonian vector field of $K_{c,\epsilon}^1$ on $T^*\mathbb{R}$ and by $\Phi^t_{K_{c,\epsilon}^2}$ the flow of the Hamiltonian vector field of $K_{c,\epsilon}^2$. We abbreviate by $S^1=\mathbb{R}/\mathbb{Z}$ the circle and define the two-dimensional torus as $T^2=S^1\times S^1$. In view of the slicing (\ref{slicing of e}) we are now in position to define a torus action
$$T^2\times\overline{\Sigma}_{c,\epsilon}^e\rightarrow \overline{\Sigma}_{c,\epsilon}^e$$
given by 
$$(t_1,t_2,z_1,w_1,z_2,w_2)\rightarrow \bigg(\Phi_{K^1_{c,\epsilon}}^{t_1\tau_{c,\epsilon}^1(K_{c,\epsilon}^1)}(z_1,w_1),\Phi_{K^2_{c,\epsilon}}^{t_2\tau_{c,\epsilon}^2(K_{c,\epsilon}^2)}(z_2,w_2)\bigg)$$

Let $\mathcal T_{c,\epsilon}^1$ be the primitive of $\tau_{c,\epsilon}^1$ given by 
$$\mathcal T^1_{c,\epsilon}(\kappa)=\int_{-\frac{M_1}{2}}^{\kappa}\tau_{c,\epsilon}^1(b)db$$
and similarly define
$$\mathcal T^2_{c,\epsilon}(\kappa)=\int_{-\frac{M_2}{2}}^{\kappa}\tau_{c,\epsilon}^2(b)db.$$

Then the map 
$$\mu_{c,\epsilon}=(\mu_{c,\epsilon}^1,\mu_{c,\epsilon}^2):\overline{\Sigma}_{c,\epsilon}\rightarrow \mathbb{R}^2=Lie(T^2)$$
with
$$\mu_{c,\epsilon}^1=\mathcal T_{c,\epsilon}^1\circ K_{c,\epsilon}^1, \mu_{c,\epsilon}^2=\mathcal T_{c,\epsilon}^2\circ K_{c,\epsilon}^2$$
is a moment map for the torus action on $\overline{\Sigma}_c$.
By the slicing (\ref{slicing of e}) its image is given by 
$$\text{im}\mu_{c,\epsilon}=\{(\mathcal T_{c,\epsilon}^1(-\kappa), \mathcal T_{c,,\epsilon}^2(\kappa))\}\subset [0,\infty)^2\subset \mathbb{R}^2.$$

The functions $\mathcal T_{c,\epsilon}^1$ and $\mathcal T_{c,\epsilon}^2$ are both strictly monotone. Therefore there exists a strictly decreasing smooth function
$$f_{c,\epsilon}:[0,\mathcal T^1_{c,\epsilon}(M_2/2)]\rightarrow [0,\mathcal T^2_{c,\epsilon}(M_1/2)]$$
such that
\begin{equation}\label{T2=f(T1)}
\mathcal T_{c,\epsilon}^2(\kappa)=f_{c,\epsilon}(\mathcal T_{c,\epsilon}^1(-\kappa)).
\end{equation}

Note that the image of the moment map can be written as the graph
$$\text{im}\mu_{c,\epsilon}=\Gamma_{f_{c,\epsilon}}.$$

Take the derivative of ($\ref{T2=f(T1)}$) with respect to $\kappa$, we have $$\tau_{c,\epsilon}^2(\kappa)=-f_{c,\epsilon}'\cdot\tau_{c,\epsilon}^1(-\kappa),$$i.e.
\begin{equation}\label{f'}
f_{c,\epsilon}'=-\frac{\tau_{c,\epsilon}^2(\kappa)}{\tau_{c,\epsilon}^1(-\kappa)}.
\end{equation}
Since $\tau_{c,\epsilon}^2(\kappa)$ and $\tau_{c,\epsilon}^1(-\kappa)$ are both positive, we have $f_{c,\epsilon}'<0$. 

The torus action and moment map on $\overline{\Sigma}_{c,\epsilon}^m$ can be defined similarly and lead to the same result as above. The only differences are the integral's boundary of $\mathcal{T}^2_{c,\epsilon}(\kappa)$ changes to $[\frac{M_2}{2}, \kappa]$ and the domain of the definition of $f_{c,\epsilon}$ is $[0, \mathcal{T}_{c,\epsilon}^1(-M_2/2)]$.

Since by Theorem 2.1 in Delzant\cite{Delzant}(see also\cite{Karshon}) the image of the moment map determines its preimage up to equivariant symplectomorphisms. Denote $\Omega=\text{im} \mu_{c,\epsilon}$, then $\mu_{c,\epsilon}^{-1}(\Omega)$ is symplectmorphism to the toric domain $X_{\Omega}$. Together with lemma \ref{lem7} we have the following theorem. We will give a brief introduction about toric domain and some of its propositions in the next section, which can also be found in \cite{Gutt}.
\begin{theorem}
Assume $0<m_2\leq m_1\leq 9m_2$, $m_1\geq\frac{\epsilon}{2}>0$ and $m_2\geq\frac{3\epsilon}{8}$, for every energy value below the minimum of the first, second and third critical value, i.e. $c<c_0$, each bounded component of the regularized energy hypersurface of the Lagrange problem arises as the boundary of a strictly monotone toric domain.
\end{theorem}

Together with Proposition 1.8 in \cite{Gutt}, also proposition 4 in the next section we can get the following corollary.

\begin{corollary}
Assume $0<m_2\leq m_1\leq 9m_2$, $m_1\geq\frac{\epsilon}{2}>0$ and $m_2\geq\frac{3\epsilon}{8}$, for every energy value below the minimum of the first, second and third critical value, i.e. $c<c_0$, each bounded component of the regularized energy hypersurface of the Lagrange problem is dynamically convex.
\end{corollary}

\section{Toric domains}

\begin{definition}
If $\Omega$ is a domain in $\mathbb{R}^n_{\geq 0}$, define the toric domain
$$X_{\Omega}=\{z\in \mathbb{C}^n|\pi(|z_1|^2,\cdot\cdot\cdot,|z_n|^2)\in\Omega\}$$
\end{definition}
The factors $\pi$ ensure that
$$Vol(X_{\Omega})=Vol(\Omega)$$
Let $\partial_+\Omega$ denote the set of $\mu\in \partial \Omega$ such that $\mu_j>0$ for all $j=1,\cdot\cdot\cdot,n.$
\begin{definition}
A strictly monotone toric domain is a compact toric domain $X_{\Omega}$ with smooth boundary such that if $\mu\in\partial_+\Omega$ and if $v$ an outward normal vector at $\mu$, then $v_j\geq 0$ for all $j=1,\cdot\cdot\cdot,n.$ 
\end{definition}

If $\Omega$ is a domain in $\mathbb{R}^n$, define 
$$\hat{\Omega}=\{|(|\mu_1|,\cdot\cdot\cdot,|\mu_n|)\in\Omega\}.$$
\begin{definition}
A convex toric domain is a  toric domain $X_\Omega$ such that $\hat{\Omega}$ is compact and convex.
\end{definition}

This terminology may be misleading because a "convex toric domain" is not the same thing as a compact toric domain that is convex in $\mathbb{R}$ as showed in the following properties.
\begin{proposition}
A toric domain $X_{\Omega}$ is a convex subset of $\mathbb{R}^{2n}$ if and only if the set 
$$\tilde{\Omega}=\{\mu\in\mathbb{R}^n|\pi(|\mu_1|^2,\cdot\cdot\cdot,|\mu_n|^2)\in\Omega\}.$$
is convex in $\mathbb{R}^n$.
\end{proposition}
\begin{proof}
See Proposition 2.3 in \cite{Gutt}.
\end{proof}
\begin{proposition}
If $X_\Omega$ is a convex toric domain, then $X_\Omega$ is a convex set of   $\mathbb{R}^{2n}.$
\end{proposition}
\begin{proof}
See Example 2.4 in \cite{Gutt}.
\end{proof}

\begin{proposition}
Let $X_\Omega$ be a compact star-shaped toric domain in $\mathbb{R}^4$ with smooth boundary. Then $X_\Omega$ is dynamically convex if and only if $X_\Omega$ is a strictly monotone toric domain.
\end{proposition}
\begin{proof}
See Proposition 1.8 in \cite{Gutt}.
\end{proof}

\section{Euler problem for one positive and one negative mass}

In \cite{Pinzari}, Gabriella Pinzari give a research on Euler problem with two fixed centers of positive masses. In the following, we generalize her results to Euler problem with the two fixed centers $e$ and $m$ of masses $m_1$ and $m_2$ satisfying $m_1>0$, $m_2\leq 0$, and $m_1>|m_2|$.  In this case, we just have to set $\epsilon=0$ in (\ref{H of Lagrange}) and the Hamiltonian function is
\begin{equation}\label{H0}
H_0(q,p)=T(p)-U_0(q)
\end{equation}
 where
$$U_0(q)=\frac{m_1}{\sqrt{(q_1+\frac{1}{2})^2+q_2^2}}+\frac{m_2}{\sqrt{(q_1-\frac{1}{2})^2+q_2^2}}$$ and $m_1> 0$, $m_2\leq 0$, $m_1> |m_2|$.
We can get the similar result for the Euler problem.

For the Euler problem with $m_1>0$ and $m_2>0$, if the energy is less than the critical value, i.e. $c<c_0$, then the Hill's region decomposes into three connected components and so is the energy hypersurface and the regularized hypersurface. Here we also abbreviate  $c_0=H_0(L)$ for Euler problem when there is no confusion.

For the Euler problem with $m_1>0$ and $m_2<0$, if $c<0$, then the Hill's region decomposes into three connected components and so is the energy hypersurface and the regularized hypersurface.  Different from the case where $m_1>0$ and $m_2>0$, the Hill's region near the fixed center $m$ with mass $m_2$ is not around but in the righthand side of $m$. If $c<c_0$, then the Hill's region consists of only one bounded component near the fixed center $e$ with mass $m_1$.

Here for the Euler problem, we only consider the bounded Hill's region around the fixed center $e$
$$\mathcal{R}_{c,0}^e=\pi(\Sigma_{c,0}^e).$$
The whole procession of regularization and the definition of toric domain are just similar as in the Lagrange problem.
After regularization, we get
$$K_{c,0}=K_{c,0}^1+K_{c,0}^2,$$
for
\begin{equation}\label{K1K2 of Euler}\begin{aligned}K_{c,0}^1=&\frac{1}{2}p_{\mu}^2-\frac{c}{4}(\cosh^2\mu-1)-\frac{M_1}{2}\cosh\mu:=\frac{1}{2}p_{\mu}^2+W_c^1(\mu).\\
K_{c,0}^2
=&\frac{1}{2}p_{\nu}^2+\frac{c}{4}(\cos^2\nu-1)+\frac{M_2}{2}\cos\nu:=\frac{1}{2}p_{\nu}^2+W_c^2(\nu).\end{aligned}\end{equation}

The reason why we define $K_c^1$ and $K_c^2$ like this is that $K_c^1$ is just the Euler integral by Lemma 3.1 in \cite{Pinzari}. It is very useful in the following. 

Set
\begin{equation}\label{x, y for Euler}
\left\{
\begin{aligned}
	x&=\cosh\mu,\\
	y&=\cos\nu.
\end{aligned}
\right.
\end{equation}
The Jacobi matrix from $(x,y)$ to $(\mu,\nu)$ is
\begin{equation}
D_2:=\frac{\partial(x,y)}{\partial(\mu,\nu)}=\left[\begin{matrix}
\frac{\partial x}{\partial \mu} & \frac{\partial x}{\partial \nu} \\
\frac{\partial y}{\partial \mu} & \frac{\partial y}{\partial \nu} 
\end{matrix}
\right]\\
=\left[
\begin{matrix} 
	\sinh\mu & 0\\ 0 &-\sinh\nu 
\end{matrix}
\right],
\end{equation}
then 
\begin{equation}\label{p_mu to p_x}
\left(\begin{matrix}p_{\mu}\\p_{\nu}\end{matrix}\right)
=D_2^{T}\left(\begin{matrix}p_x\\p_y\end{matrix}\right)
=\left[
\begin{matrix} 
	\sinh\mu & 0\\ 0 &-\sinh\nu 
\end{matrix}
\right]\left(\begin{matrix}p_x\\p_y\end{matrix}\right),\end{equation}
then we have $$K_{c,0}=\frac{1}{2}(x^2-1)P_x^2+\frac{1}{2}(1-y^2)P_y^2-\frac{c}{4}x^2+\frac{c}{4}y^2-\frac{M_1}{2}x+\frac{M_2}{2}y.$$
and
$$\begin{aligned}K_{c,0}^{1}:&=\frac{x^2-1}{2}p_x^2+V_c^1=\frac{x^2-1}{2}p_x^2-\frac{c}{4}(x^2-1)-\frac{M_1}{2} x,\\
K_{c,0}^{2}:&=\frac{1-y^2}{2}p_y^2+V_c^2=\frac{1-y^2}{2}p_y^2+\frac{c}{4}(y^2-1)+\frac{M_2}{2} y.\end{aligned}$$

By Legendre transformation, we know that $$p_x=\frac{1}{x^2-1}\dot{x}, p_y=\frac{1}{1-y^2}\dot{y},$$ where $\dot{x}=\frac{dx}{dt}, \dot{y}=\frac{dy}{dt}$.

Similar as in the Lagrange problem, by analyzing the potential energy of the separated systems we know that and if the condition $c+M_1<0$ holds, then the graphs of $W_c^1(\mu)$ and $W_c^2(\nu)$ are like cups near the fixed centers.

For he Euler problem with $m_1>0$ and $m_2>0$, since $c_0<-M_1$, when $c<c_0$, the bounded regularized energy surface around $e$ is $\overline \Sigma_{c,0}^e$ and
\begin{equation}\label{slicing of Euler}
\overline{\Sigma}^e_{c,0}=\bigcup_{\kappa\in[\kappa^{1,e}_{c,0},\kappa^{2,e}_{c,0}]}S_{-\kappa,c,0}^{1,e}\times S_{\kappa,c,0}^{2,e}\end{equation}

where $S_{-\kappa,c,0}^{1,e}\times S_{\kappa,c,0}^{2,e}$ is an Arnold-Liouville torus for $-\frac{M_2}{2}<\kappa<-\frac{M_1}{2}$ and the Arnold-Liouville torus degenerates to a circle for $\kappa=-\frac{M_2}{2}$ or $\kappa=-\frac{M_1}{2}$. 

For the Euler problem with $m_1>0$ and $m_2<0$, when $c<-M_1$, we have the same result as above. Here we also have $c_0<-M_1$.

When $K_{c,0}^1=-\kappa$, $K_{c,0}^1=\kappa$, and the periods are as following.
\begin{equation}\label{tau 1}
\begin{aligned}
\tau^1_{c}(-\kappa):=\tau^1_{c,0}(-\kappa)&=4\int_0^{\frac{\tau^1_{c,\epsilon}(-\kappa)}{4}}dt\\
&=4\int_1^{x_0}\frac{dx}{\sqrt{2(x^2-1)\big(-\kappa+\frac{c}{4}(x^2-1)+\frac{M_1}{2} x\big)}}	
\end{aligned}
\end{equation}
\begin{equation}\label{tau 2}
\begin{aligned}
\tau^2_{c}(\kappa):=\tau^2_{c,0}(\kappa)&=4\int_0^{\frac{\tau^2_{c,\epsilon}(\kappa)}{4}}dt\\
&=4\int_{y_0}^1\frac{dy}{\sqrt{2(1-y^2)\big(\kappa-\frac{c}{4}(y^2-1)-\frac{M_2}{2}y\big)}}\\	
\end{aligned}
\end{equation}

Define
\begin{equation}\label{W}
W(\kappa)=\frac{\tau_c^2(\kappa)}{\tau_c^1(-\kappa)}=\frac{\tau(M_2,\kappa)}{\tau(M_1,\kappa)},
\end{equation}
By (\ref{f'}), we get $$f'_{c,0}(\mathcal T_c^1(-\kappa))=-W(\kappa).$$
Take its derivative with respect to $\kappa$, we have
\begin{equation}\label{f'' for Euler}f''_{c,0}(\mathcal T_c^1(-\kappa))=\frac{\partial_{\kappa}W(\kappa)}{\tau_c^1(-\kappa)}.
\end{equation}

Since by Theorem 2.1 in Delzant\cite{Delzant}(see also\cite{Karshon}) the image of the moment map determines its preimage up to equivariant symplectomorphisms. The sign of (\ref{f'' for Euler}) determines the convexity of the toric domain. Because $\tau_c^1(\kappa)>0$, we only need to determine the sign of $\partial_{\kappa}W(\kappa)$. But at first we need $\tau_c^1(-\kappa)$ and $\tau_c^2(\kappa)$ in a kind of good expression to allow us to achieve. 

As Gagriella Pinzari observed in Theorem 1.2 of \cite{Pinzari}, for Euler problem with $m_1>m_2>0$, the periods $\tau^1_{c}(-\kappa)$ and $\tau^2_{c}(\kappa)$ of the two separated systems of the Euler problem only depend on $M_1$ and $M_2$ respectively, not on the exact value of $m_1$ and $m_2$ at all. So we can choose $m_2=0$, then $\tau^1_{c}(-\kappa)$ is equal to the period of the first separated system of the regularized Kepler problem with centre mass $M_1$ and $\tau^2_{c}(\kappa)$ is equal to the period of the second separated system of the regularized Kepler problem with centre mass $M_2$. Since the two periods of the separated systems of the regularized Kepler problem with a fixed mass are the same, we can say that $\tau^1_{c}(-\kappa)$ and $\tau^2_{c}(\kappa)$ are equal to the periods of the separated system of the regularized Kepler problem with center masses $M_1$ and $M_2$ respectively.

This is also true for Euler problem with $m_1> 0$, $m_2\leq 0$ and $m_1> |m_2|$. So we can use the method of Gagriella Pinzari \cite{Pinzari} here.

The period of the regularized Kepler problem is also related to the period of the original Kepler problem which can be computed exactly and appears in a kind of good expression under the parameter eccentric anomaly for us to determine the sign of (\ref{f'' for Euler}) and then the convexity of the toric domain of the Euler problem with $m_1> 0$, $m_2\leq 0$ and $m_1> |m_2|$ by the following four steps.

Step 1: About the precise relationship between the periods of the Kepler problem and the regularized Euler problem.

In (\ref{H0}), we put the origin of Euler problem in the middle of the two big masses. But If we want to go to the Kepler problem to find the period of the regularized Euler problem, it is more easy for us to compute if we put the origin of the Cartesian coordinate on the nonzero mass.

We choose the origin of the Cartesian coordinate to be the body $e$ with mass $\tilde{m}_1$, then the Hamiltonian of the Euler problem is 

\begin{equation}\label{H of Euler}
H(\tilde{q},\tilde{p})=T(\tilde{q},\tilde{p})-U(\tilde{q},\tilde{p})
\end{equation}
 where
$$U(\tilde{q},\tilde{p})=\frac{\tilde{m}_1}{\sqrt{\tilde{q}_1^2+\tilde{q}_2^2}}+\frac{\tilde{m}_2}{\sqrt{(\tilde{q}_1-1)^2+\tilde{q}_2^2}}.$$
When $\tilde{m}_2=0$, $H$ is just the Hamiltonian of Kepler problem.

The relationship between the initial coordinate $(\tilde{q},\tilde{p})$ used here  and the coordinate $(q,p)$ in Euler problem are
\begin{equation}\label{tilde q1,q2}
\tilde q_1=q_1+\frac{1}{2}, \tilde q_2=q_2.
\end{equation}

Actually, we just changed the coordinate horizontally to move the point $(-\frac{1}{2}, 0)$ to the origin. 

By the transformation of elliptic coordinates (\ref{q1,q2}), we can get
\begin{equation}\left\{
\begin{aligned}
(q_1+\frac{1}{2})^2+q_2^2=&\frac{1}{4}(\cosh\mu+\cos\nu)^2\\
(q_1-\frac{1}{2})^2+q_2^2=&\frac{1}{4}(\cosh\mu-\cos\nu)^2.
\end{aligned}
\right.
\end{equation}
Together with (\ref{x, y for Euler}) and (\ref{tilde q1,q2}), we have
\begin{equation}\label{x, y, tilde q1, q2}\left\{
\begin{aligned}
&\tilde{q}_1^2+\tilde{q}_2^2=\frac{1}{4}(x+y)^2\\
&(\tilde{q}_1-1)^2+\tilde{q}_2^2=\frac{1}{4}(x-y)^2.
\end{aligned}
\right.
\end{equation}

By (\ref{H to K}) and (\ref{x,y}), in the process of regularization using elliptic coordinate there is a time rescaling from the coordinate (q(t),p(t)) to coordinate $(x(\tau),y(\tau))$, also a rescaling from $(\tilde{q}(t),\tilde{p}(t))$ to $(x(\tau),y(\tau))$ considering the extra horizontal shift as we mentioned above.
Their relationship is 

$$d\tau=\frac{4}{x^2-y^2}dt$$
Using (\ref{x, y, tilde q1, q2}), we get

\begin{equation}
\begin{aligned}
\tau_{\tilde{m}_1,\tilde{m}_2}(t)=&\int_0^t{\frac{4}{x^2-y^2}}dt'\\
=&\int_0^t{\frac{4}{\sqrt{(\tilde{q}_1^2+\tilde{q}_2^2)((\tilde{q}_1-1)^2+\tilde{q}_2^2)}}}dt'
\end{aligned}
\end{equation}

Assume $\tilde{m}_1=M_1$ and $\tilde{m}_2=0$, then $(\tilde q,\tilde p)$ is a solution of Kepler problem with Hamiltonian function
\begin{equation}
H(\tilde{q},\tilde{p})=T(\tilde{q},\tilde{p})-U(\tilde{q},\tilde{p})
\end{equation}
 where
$$U(\tilde{q},\tilde{p})=\frac{M_1}{\sqrt{\tilde{q}_1^2+\tilde{q}_2^2}}$$
and we can compute the period in elliptic coordinate by 
\begin{equation}\label{tau 1 (kappa)}
\tau_{c}^1(-\kappa)=\tau_{M_1,0}(T):=\tau_{M_1}(T),
\end{equation}
where $T$ is the periodic of the orbit in the origin coordinate.
Analogously, assume $\tilde{m}_1=M_2$ and $m_2=0$, $(\tilde q,\tilde p)$ is a solution of Kepler problem with Hamiltonian function
\begin{equation}\label{H?}
H(\tilde{q},\tilde{p})=T(\tilde{q},\tilde{p})-U(\tilde{q},\tilde{p})
\end{equation}
 where
$$U(\tilde{q},\tilde{p})=\frac{M_2}{\sqrt{\tilde{q}_1^2+\tilde{q}_2^2}}$$
and
\begin{equation}\label{tau 2 (kappa)}
\tau_{c}^2(\kappa)=\tau_{M_2,0}(T):=\tau_{M_2}(T).
\end{equation}

Step2: About the Euler integral used to compute the period of the Kepler problem.

In the elliptic coordinate $\tau_{c}^1(\kappa)$ and $\tau^2_{c}(\kappa)$ only depend on $c$ and $\kappa$ when the masses are fixed. While in the initial coordinate $c$ and $\kappa$ are the total energy and the Euler integral of the Kepler problem respectively. Given these two integrals, we can know the orbit of Kepler problem exactly under any initial condition. Using the information of the orbit, we can compute $\tau_{M_1}(2\pi)$ and $\tau_{M_2}(2\pi)$.

For Euler problem with Hamiltonian (\ref{H of Euler}) in the original coordinate, define the Euler integral as 
$$E=||L||^2-e_1\cdot(\tilde{p}\times L-\tilde{m}_1\frac{\tilde{q}}{||\tilde{q}||}+\tilde{m}_2\frac{\tilde{q}-e_1}{||\tilde{q}-e_1||})$$
where $e_1=(1,0)$ and $L$ is the angular momentum
$$L=\tilde{q}\times \tilde{p}.$$

As in Lemma 3.1 in (\cite{Pinzari}), we now show that $E=K_{c,0}^2$.

$$L=\tilde{q}\times\tilde{q}=\tilde{q}\times p=q_1p_2-q_2q_1+\frac{1}{2}p_2.$$
Then
$$\begin{aligned}||L||^2-e_1\cdot (p\times L)=&||L||^2-e_1\cdot (\tilde{p}\times L)\\=&(q_1p_2-q_2q_1+\frac{1}{2}p_2)^2-p_2(q_1p_2-q_2p_1+\frac{1}{2}p_2)\\=&(q_1^2-\frac{1}{4})p_2^2+q_2^2p_1^2-2q_1q_2p_1p_2.\end{aligned}$$
By (\ref{q1,q2}), (\ref{x, y for Euler}), (\ref{p_mu to p_x}) and (\ref{p_1 to p_mu}), under a direct computation, we get
$$||L||^2-e_1\cdot (p\times L)=\frac{(x^2-1)(1-y^2)}{(x^2-y^2)}(p_y^2-p_x^2).$$
By (\ref{x, y, tilde q1, q2})
$$e_1\cdot\tilde{m}_1\frac{\tilde{q}}{||\tilde{q}||}=\tilde{m}_1\frac{xy+1}{x+y},$$
$$e_1\cdot\tilde{m}_2\frac{\tilde{q}-e_1}{||\tilde{q}-e_1||}=\tilde{m}_2\frac{xy-1}{x-y}.$$
Therefore, $$\begin{aligned}E=&||L||^2-e_1\cdot (p\times L)\\=&\frac{(x^2-1)(1-y^2)}{(x^2-y^2)}(p_y^2-p_x^2)+\tilde{m}_1\frac{xy+1}{x+y}-\tilde{m}_2\frac{xy-1}{x-y}\\=&\frac{x^2-1}{x^2-y^2}K_{c,0}^2-\frac{1-y^2}{x^2-y^2}K_{c,0}^1\\=&\frac{x^2-1}{x^2-y^2}\kappa-\frac{1-y^2}{x^2-y^2}(-\kappa)\\=&\kappa\end{aligned}.$$

For Kepler problem with Hamiltonian 
\begin{equation}
H(\tilde{q},\tilde{p})=T(\tilde{q},\tilde{p})-U(\tilde{q},\tilde{p}),
\end{equation}
 where
$$U(\tilde{q},\tilde{p})=\frac{M}{\sqrt{\tilde{q}_1^2+\tilde{q}_2^2}},$$
the Euler integral become
\begin{equation}\label{Euler integral}
E=||L||^2-e_1\cdot A,\end{equation}
where $A$ is the Runge-Lenz vector
$$A=\tilde{p}\times L-M \frac{\tilde{q}}{||\tilde{q}||}.$$
and since $L$ and $A$ are both first integrals for the Kepler problem, $E$ is also a first integral.

For Kepler problem, fix the energy $c$ and the Euler integral E, given the initial condition $(\tilde{q},\tilde{p})=(\tilde{q}_0,\tilde{p}_0)$, there is a unique elliptic orbit $\mathcal{O}$ going through $(\tilde{q}_0,\tilde{p}_0)$. Let $P$ be the vector perihelion of this orbit, $\nu$ be the angular from $e_1$ to $P$. we know that 
$$A=MeP.$$ Define
$$\omega=\nu+\frac{\pi}{2}.$$
Actually, we rotate the vector $P$ about the origin for $\frac{\pi}{2}$ to get a vector $n$, then $\omega$ is just the angular from $e_1$ to $n$.
Together with the following (\ref{r}) and (\ref{r from theta}), the Euler integral can be rewritten as
\begin{equation}\label{Euler integral 1}
E=M(a(1-e^2)-e\sin \omega).
\end{equation}
where $a$ is the major semi axis.

Step 3: Computation of the period of the Kepler problem.

With these preparation, we are able to compute $\tau_M(T)$ by changing its parameter from $t$ to $\theta$ and then to $\xi$.
Let $\theta$ be the true anomaly, $\xi$ the eccentric anomaly of orbit $\mathcal{O}$, 
then
\begin{equation}
\left\{
\begin{aligned}
\tilde{q}_1=&a\cos(\theta+\nu)=a\sin(\theta+\omega),\\
\tilde{q}_2=&b\sin(\theta+\nu)=-b\cos(\theta+\omega).
\end{aligned}
\right.
\end{equation}
The relationship between $\theta$ and $\xi$ is 
\begin{equation}\label{xi to theta}
\cos\xi=\frac{e+\cos\theta}{1+e\cos\theta}, \sin\xi=\frac{\sqrt{1-e^2}\sin\theta}{1+e\cos\theta},
\end{equation}
and
\begin{equation}\label{theta to xi}
\cos\theta=\frac{\cos\xi-e}{1-e\cos\xi}, \sin\theta=\frac{\sqrt{1-e^2}\sin\xi}{1-e\cos\xi},
\end{equation}
where $e$ is the eccentricity of the orbit $\mathcal{O}$.
As a result,
\begin{equation}\label{d theta to d xi}
d\theta=\frac{\sqrt{1-e^2}}{1-ecos\xi}d\xi
\end{equation}
and
\begin{equation}\label{r}
r=a(1-ecos\xi)=\frac{a(1-e^2)}{1+ecos\theta}.
\end{equation}
where $r=\sqrt{\tilde{q}^2_1+\tilde{q}^2_2}$, $a$ is the major semi axis and $b$ is the minor semi axis.
By \cite{Arnold}, we know that
\begin{equation}\label{r from theta}
r=\frac{A^2/M}{1+ecos\theta},
\end{equation}
the Angular momentum 
\begin{equation}\label{A}
A=r^2\frac{d\theta}{dt}
\end{equation}
is a constant,
and $a$ only depends on the mass and the total energy c.
\begin{equation}\label{a from c}
a=\frac{M}{2|c|}.
\end{equation}
By (\ref{r}), (\ref{r from theta}) and (\ref{A}), we can infer that 
\begin{equation}\label{A to a}
A=\sqrt{Ma(1-e^2)},
\end{equation}
and
\begin{equation}\label{t to theta}
dt=\frac{r^2}{A}d\theta.
\end{equation}

\begin{equation}
\tau_M(T)=\int_0^{T}\frac{4}{\sqrt{\big(\tilde{q}_1^2(t)+\tilde{q}_2^2\big)\big((\tilde{q}_1(t)-1)^2+\tilde{q}_2^2(t)\big)}}dt.
\end{equation}
We know that
\begin{equation}
\tilde{q}_1^2(t)+\tilde{q}_2^2(t)=r^2.
\end{equation}

$$(\tilde{q}_1(t)-1)^2+\tilde{q}_2^2(t)=r^2-2a\sin(\theta+\omega)+1$$
Together with (\ref{t to theta}),
$$\tau_M(T)=4\int_0^{2\pi}\frac{r}{aA\sqrt{r^2-2a\sin(\theta+\omega)+1}}d\theta.$$
By (\ref{theta to xi}),
$$\begin{aligned}&r^2-2a\sin(\theta+\omega)+1\\=&a^2\big((1-e\cos\xi)^2-\frac{2}{a}(\sqrt{1-e^2}\sin\xi\cos\omega+(\cos\xi-e)\sin\omega)+\frac{1}{a^2}\big)\end{aligned}$$
Together with (\ref{d theta to d xi}),

$$\begin{aligned}&\tau_M(T)\\=&4\int_0^{2\pi} r(1-e^2)^{\frac{1}{2}}a^{-1}A^{-1} (1-e\cos\xi)^{-1}\\ &\cdot \big((1-e\cos\xi)^2-\frac{2}{a}(\sqrt{1-e^2}\sin\xi\cos\omega+(\cos\xi-e)\sin\omega)+\frac{1}{a^2}\big)^{-\frac{1}{2}} d\xi.\end{aligned}$$
By (\ref{A to a}) and (\ref{r}),

$$\begin{aligned}&\tau_M(T)\\=&4\int_0^{2\pi}
(Ma)^{-\frac{1}{2}}\\
&\cdot\big((1-e\cos\xi)^2-\frac{2}{a}(\sqrt{1-e^2}\sin\xi\cos\omega+(\cos\xi-e)\sin\omega)+\frac{1}{a^2}\big)^{-\frac{1}{2}} d\xi.\end{aligned}$$
By(\ref{a from c}),

\begin{equation}\label{tau from xi}
\begin{aligned}
&\tau_M(T)\\=&4\sqrt{2|c|}\int_0^{2\pi}\big(M^2(1-e\cos\xi)^2\\&-4|c|M(\sqrt{1-e^2}sin\xi\cos\omega+(\cos\xi-e)\sin\omega)+4c^2\big)^{-\frac{1}{2}} d\xi.
\end{aligned}
\end{equation}

We can simplify (\ref{tau from xi}) using the Euler integral in the following.
Given the energy $H=c$ and Euler integral $E=\kappa$,
\begin{equation}\label{Euler integral 2}
M(a(1-e^2)-e\sin \omega)=\kappa,
\end{equation}
together with (\ref{a from c}),
we have
$$M(\frac{M}{2|c|}(1-e^2)-e\sin \omega)=\kappa.$$
For the orbit with eccentricity $e=1$
From (\ref{tau 1}) and (\ref{tau 2}), we can find that the period only depend on the energy $c$ and the Euler integral $\kappa$, not depend on the shape of the exact orbit given an initial condition, namely, not depend on the eccentricity and the angular $\omega$, so we can just choose $e=1$ to get a simpler form of (\ref{tau from xi}). Plug $e=1$ into  (\ref{Euler integral 2}), we get
\begin{equation}
\sin\omega=-\frac{\kappa}{M}.
\end{equation} 
Since $-\frac{M_2}{2}\leq\kappa\leq\frac{M_1}{2}$, while for the Kepler problem with mass $M$, we have $-\frac{M}{2}\leq\kappa\leq\frac{M}{2}$, therefore, \begin{equation}\label{assumption}
|\frac{\kappa}{M}|\le1\end{equation} 
This shows that $\omega$ is sensiable. Then (\ref{tau from xi}) becomes
\begin{equation}\label{tau 3}
\tau_M(T)=4\sqrt{2|c|}\int_0^{2\pi}\frac{d\xi}{\sqrt{M^2(1-\cos\xi)^2-4|c|\kappa(1-\cos\xi)+4c^2}}.
\end{equation}
Set $z=1-\cos\xi$,
then (\ref{tau 3}) finally becomes
\begin{equation}\label{tau 4}
\tau_M(T)=4\sqrt{2|c|}\int_0^{2}\frac{dz}{\sqrt{z(2-z)(M^2z^2-4|c|\kappa z+4c^2)}}.
\end{equation}
For $c<0$, this is just
\begin{equation}\label{tau 5}
\tau_M(T)=4\sqrt{-2c}\int_0^{2}\frac{dz}{\sqrt{z(2-z)(M^2z^2+4c\kappa z+4c^2)}}.
\end{equation}
By (\ref{tau 1 (kappa)}) and (\ref{tau 2 (kappa)}), we have
$$\tau_{c}^1(-\kappa)=4\sqrt{-2c}\int_0^{2}\frac{dz}{\sqrt{z(2-z)(M_1^2z^2+4c\kappa z+4c^2)}}.$$
and
$$\tau_{c}^2(\kappa)=4\sqrt{-2c}\int_0^{2}\frac{dz}{\sqrt{z(2-z)(M_2^2z^2+4c\kappa z+4c^2)}}.$$

Note
$$\tau(M, \kappa)=4\sqrt{-2c}\int_0^{2}\frac{dz}{\sqrt{z(2-z)(M^2z^2+4c\kappa z+4c^2)}}.$$
then $\tau_c^1(-\kappa)=\tau(M_1,\kappa)$, $\tau_c^2(\kappa)=\tau(M_2,\kappa)$.

Step 4: Computation of the sign of (\ref{f'' for Euler}).
Note \begin{equation}\label{A,B,C}A=M^2, B=-2c\kappa, C=4c^2,\end{equation}
then we have
$$\tau(M, \kappa)=4\sqrt{-2c}\int_0^{2}\frac{dz}{\sqrt{z(2-z)(Az^2-2Bz+C)}}$$
and
$$\partial_\kappa W(\kappa)=-2c\cdot\partial_B W(\kappa).$$  

Since the energy $c<0$, $\partial_\kappa W(\kappa)$ and $\partial_B W(\kappa)$ have the same sign. We know that $\partial_B W(\kappa)$ and $\partial_B \ln W(\kappa)$ also have the same sign because of 
$\partial_B \ln W(\kappa)=\frac{\partial_B W(\kappa)}{W(\kappa)}$
and $W(\kappa)>0$. 
\begin{equation}
\begin{aligned}
\partial_B \ln W(\kappa)
=&\partial_B \ln \frac{\tau(M_2,\kappa)}{\tau(M_1,\kappa)}\\
=&\partial_B \ln \tau(M_2,\kappa)-\partial \ln \tau(M_1,\kappa)\\
=&\frac{\partial_B\tau(M_2,\kappa)}{\tau(M_2,\kappa)}-\frac{\partial_B\tau(M_1,\kappa)}{\tau(M_1,\kappa)}
\end{aligned}
\end{equation}

Define a function $\eta(M,\kappa)$
$$\eta(M,\kappa):=\frac{\partial_B\tau(M,\kappa)}{\tau(M,\kappa)}$$
then $$\partial_B \ln W(\kappa)=\eta(M_2,\kappa)-\eta(M_1,\kappa).$$
$\partial_B \ln W(\kappa)$ is definitely positive or negative if $\eta(M,\kappa)$ is a monotonic function with respect to $M$. Since $M>0$ and $A=M^2$, the monotonicity of the function $\eta$ with respect to the variable $A$ are the same as that of $M$. Here we can also denote $\eta(M,\kappa)$ by $\eta(A,B)$ and denote $\tau(M,\kappa)$ by $\tau(A,B)$.
$$\eta(A,B)=\frac{\partial_B\tau(A,B)}{\tau(A,B)}$$
\begin{equation}\label{partial eta}
\partial_A \eta(A, B)=-\frac{\partial_B\tau(A,B)\cdot\partial_A\tau(A,B)-\partial_A\partial_B\tau(A,B)\cdot\tau(A,B)}{\tau(A,B)^2}
\end{equation}
Set
$$Q(A,B,C,x):=Ax^2-2Bx+C,$$
$$f_{\alpha}^{\beta}(A,B,C,x):=\frac{x^{\beta}}{Q(A,B,C,x)^{\alpha}\sqrt{2-x}},$$
$$\begin{aligned}g_{\alpha}^{\beta}(A,B,C,x ):=&4\sqrt{-2c}\int_0^2 f_{\alpha}^{\beta}(A,B,C,x)dx\\=&4\sqrt{-2c}\int^2_0\frac{x^{\beta}}{Q(A,B,C,x)^{\alpha}\sqrt{2-x}}dx\end{aligned}$$
Note the numerator of (\ref{partial eta}) as $S(A,B,C)$
\begin{equation}\label{S}
\begin{aligned}
S(A,B,C)
=&\partial_B\tau(A,B)\cdot\partial_A\tau(A,B)-\partial_A\partial_B\tau(A,B)\cdot\tau(A,B)\\
=&\frac{1}{2}\big(3g_{\frac{5}{2}}^{\frac{5}{2}}g_{\frac{1}{2}}^{-\frac{1}{2}}(A,B,C)-g_{\frac{3}{2}}^{\frac{3}{2}}(A,B,C)g_{\frac{3}{2}}^{\frac{1}{2}}(A,B,C)\big).
\end{aligned}
\end{equation}
Let
$$p(A,B,C,x):=f_{\frac{1}{2}}^{-\frac{1}{2}}(A,B,C,x),$$
then
$$f_{\frac{5}{2}}^{\frac{5}{2}}(A,B,C,x)=\frac{x^3}{Q(A,B,C,x)^2}p(A,B,C,x),$$
$$f_{\frac{3}{2}}^{\frac{3}{2}}(A,B,C,x)=\frac{x^2}{Q(A,B,C,x)^2}p(A,B,C,x),$$
$$f_{\frac{3}{2}}^{\frac{1}{2}}(A,B,C,x)=\frac{x}{Q(A,B,C,x)^2}p(A,B,C,x).$$
and  (\ref{S}) becomes
$$\begin{aligned}S(A,B,C)=&\frac{1}{2}\int_0^2\int_0^2\bigg(\frac{3x^3}{Q(A,B,C,x)^2}-\frac{x^2y}{Q(A,B,C,x)Q(A,B,C,y)}\bigg)\\&p(A,B,C,x)p(A,B,C,y)dxdy\end{aligned}$$
Since the functions $x\rightarrow \frac{x}{Q(A,B,C,x)}$ and $x\rightarrow\frac{x^2}{Q(A,B,C,x)}$ have the same monotonicity on $x\in[0,2]$.
Using the Chebyshev integral inequality in proposition 4, we can find $S>0$.
Note $Q(x)=Q(A,B,C,x)$, $p(x)=p(A,B,C,x)$ and $S=S(A,B,C)$ for short, indeed, 
\begin{equation}
\begin{aligned}
S=&\frac{1}{2}\int^2_0\int^2_0\bigg(\frac{3x^3}{Q(x)^2}-\frac{x^2y}{Q(x)Q(y)}\bigg)p(x)p(y)dxdy\\
=&\frac{1}{2}\int^2_0\frac{3x^3}{Q(x)^2}p(x)dx\int^2_0p(y)dy-\frac{1}{2}\int^2_0\frac{x^2}{Q(x)}p(x)dx\int^2_0\frac{y}{Q(y)}p(y)dy\\
\geq&\frac{1}{2}\int^2_0\frac{3x^3}{Q(x)^2}p(x)dx\int^2_0p(y)dy-\frac{1}{2}\int_0^2\frac{x^3}{Q(x)^2}p(x)dx\int_0^2p(y)dy\\
=&\int^2_0\frac{x^3}{Q(x)^2}p(x)dx\int^2_0p(y)dy\\
>&0
\end{aligned}
\end{equation}
\begin{proposition} (Chebyshev Integral Inequality)
let $f$, $g$, $p:\mathbb{R}\rightarrow\mathbb{R}$ with $p\ge 0$, $p$, $fp$, $gp$, $fgp$ are integrable on $\mathbb{R}$, $\int_{\mathbb{R}}p(x)=1$, $f$ and $g$ are both decreasing or increasing on the support of $p$, then 
\begin{equation}\label{Chebyshev1}
\bigg(\int_{\mathbb{R}}f(x)p(x)dx\bigg)\bigg(\int_{\mathbb{R}}g(y)p(y)dy\bigg)\le\int_{\mathbb{R}}f(x)g(x)p(x)dx.
\end{equation}
Releasing the assumption $\int_{\mathbb{R}}p(x)=1$, (\ref{Chebyshev1}) is replaced by
\begin{equation}\label{Chebyshev2}
\bigg(\int_{\mathbb{R}}f(x)p(x)dx\bigg)\bigg(\int_{\mathbb{R}}g(y)p(y)dy\bigg)\le\bigg(\int_{\mathbb{R}}f(x)g(x)p(x)dx\bigg)\bigg(\int_{\mathbb{R}}p(y)dy\bigg).
\end{equation}
\begin{proof}
As $f$, $g$ are decreasing or increasing on the support of p, for any $x$, $y$ on such support, we have 
$$0\leq (f(x)-f(y))(g(x)-g(y))=f(x)g(x)-f(x)g(y)-f(y)g(x)-f(y)g(y).$$
Multiplying by $p(x)p(y)$ and taking the integral on $\mathbb R^2$ we get the proposition.
\end{proof}
\end{proposition}
Since $S>0$, we get $\partial_A \eta(A, B)<0$. As a result, $\eta(A,B)$ is a decreasing function. Since $m_1>0$, $m_2\leq 0$ and $M_1=m_1+m_2$, $M_2=m_1-m_2$, we have $M_1<M_2$, then $\partial_B \ln W(\kappa)<0$, finally, 
\begin{equation}\label{partial W}
\partial_{\kappa}W(\kappa)<0.
\end{equation}
 This result is just opposite to the case when $m_1$ and $m_2$ are both positive in  \cite{Pinzari}. 

By (\ref{partial W}) and $\tau_c^1(-\kappa)>0$, we have 
$$f''_{c,0}(\mathcal T_c^1(-\kappa))<0.$$

Under the four steps above, we have proved the following theorem. 

\begin{theorem}
When the energy $c<-M_1$, the bounded component around the fixed center $e$ of the regularized energy hypersurface of Euler problem with two fixed centers $e$ and $m$ of masses $m_1$ and $m_2$ respectively satisfying $m_1>0, m_2\leq 0, m_1\geq|m_2|$ arises as a convex toric domain.
\end{theorem}
Combining Gabriella Pinzari's result in \cite{Pinzari} and theorem 2, we have the following corollary for Euler problem.
\begin{corollary}
When the energy is less than the critical value, i.e. $c<c_o$, the toric domain $X_{\Omega_{m_2}}$ defined for the bounded component around the fixed center $e$ of the regularized energy hypersurface of the Euler problem  satisfying $m_1>0, m_2\leq 0, m_1\geq|m_2|$ is concave for $m_2\geq0$, convex for $m_2\leq 0$.
\end{corollary}

\section*{Acknowledgments}
I want to express my appreciation to Urs Frauenfelder for his suggestions and discussions on this paper during and after I visited Augsburg University.

\end{document}